\documentclass[10pt]{article}

\usepackage[pagebackref=true, colorlinks=true, linkcolor=blue, citecolor=blue]{hyperref}

\usepackage{amssymb,amsmath,graphicx,amsfonts,euscript,fullpage}
\usepackage{fancyhdr}
\usepackage{epsfig}
\usepackage{setspace}
\usepackage[all]{xy}
\usepackage{amsmath}
\usepackage{graphicx}
\usepackage{caption}
\usepackage{subcaption}
\usepackage{multirow}
\usepackage{enumitem}
\usepackage{authblk}
\usepackage{amsthm}

 \usepackage{pst-grad}
 \usepackage{pst-plot}
\def \bes{\begin{eqnarray}}
\def \ees{\end{eqnarray}}

\newtheorem{thm}{Theorem}

\numberwithin{equation}{section}

\title{A randomized Newton's method for solving differential equations based on the neural network discretization}

\author[$1$]{Qipin Chen}
\author[$1$]{Wenrui Hao}

\affil[$1$]{Department of Mathematics, Pennsylvania State University, University Park, PA 16802}
\begin{document}

\maketitle

\begin{abstract}
We develop a randomized Newton's method for solving differential
equations, based on a fully connected neural network discretization.
In particular, the randomized Newton's method randomly chooses  equations from the overdetermined nonlinear system resulting from the neural network discretization and solves the nonlinear system adaptively. We prove theoretically that the randomized Newton's method has a quadratic convergence locally.
We also apply this new method to
various numerical examples, from one- to high-dimensional differential equations, in order to verify  its feasibility and efficiency. Moreover, the randomized Newton's method can allow the neural network to "learn" multiple solutions for nonlinear systems of differential equations, such as pattern formation problems, and provides an alternative way to study the solution structure of nonlinear differential equations overall.
\end{abstract}


\section{Introduction}
Partial differential equations (PDEs) have been widely used in  physics \cite{courant2008methods}, biology \cite{fedosov2011multiscale, lei2012quantifying}, and engineering \cite{ames1965nonlinear},
being used to model everything from bacterial growth \cite{baranyi1993non} to complex fluid structure interactions \cite{paidoussis1998fluid}.  Thus, computing numerical solutions for PDEs has been a research area of long-standing importance in the computational mathematics community. Some efficient numerical methods have already been developed for solving PDEs: for example, the finite difference method \cite{xing2005high,  shu2003high}, the finite element method \cite{xu1999monotone, marion1995error} and the spectral method \cite{shen2011spectral, xiu2010numerical, karniadakis2013spectral, schumack1991spectral}. Moreover, several numerical techniques, such as
the multigrid methods \cite{bramble1991convergence, xu2004energy}, domain decomposition methods  \cite{toselli2006domain, smith2004domain}, and preconditioning techniques  have been used to speed up computations and improve computational efficiency, especially in two- and three-dimensional problems.
However, there are still two challenges facing the numerical PDE community: 1) traditional methods become inefficient for solving high-dimensional PDEs, due to such PDEs' dramatic explosion of grid points. While the sparse grid method \cite{nobile2008sparse} has been used to solve high-dimensional PDEs \cite{shen2010efficient, liu2019krylov} by a constructing multidimensional, multilevel basis, this method still becomes inefficient when the dimension is particularly high, owing to the logarithmic term in the complexity \cite{bungartz2004sparse}; 2) traditional methods are inadequate for computing the multiple solutions that nonlinear PDEs have. While the deflation method \cite{farrell2015deflation} has been used to compute the distinct solutions, it can not be guaranteed to find all the possible solutions, due to the artificial singularities that this  method introduces. And while homotopy methods coupled with domain decomposition \cite{hao2014bootstrapping}, multigrid and spectral methods \cite{wang2018two} have been developed for computing multiple solutions, all of these methods become time-consuming for high-dimensional, nonlinear PDEs.

Recently, machine learning techniques have been developed for solving PDEs, since machine learning has been experiencing great success in various fields related to artificial intelligence (e.g., computer vision \cite{forsyth2002computer}, natural language processing \cite{manning1999foundations}). The application of machine learning techniques to PDE problems, however, is usually not straightforward. Some approaches have included the following: The DGM net \cite{sirignano2018dgm}, based on a fully connected network, has been developed to solve high-dimensional PDEs by minimizing the $L^2$ norm, both in the domain and on the boundary; Using ReLU deep neural networks (DNNs)  has been developed in \cite{he2018relu} to solve differential equations by exploring
the relationship between DNNs with rectified linear unit (ReLU) function and continuous piecewise linear functions, from finite element method;
A deep learning-based approach \cite{weinan2017deep} has been developed to solve high-dimensional parabolic PDEs by reformulating PDEs as backward stochastic differential equations; Machine learning techniques have also been  used to learn governing differential equations by empirical data \cite{qin2019data, wu2019data, wang2019efficient}. All these approaches follow the machine learning optimization framework by minimizing the loss functions constructed according to different approaches. However, these loss functions are usually highly non-convex, and the optimization process is prone to being trapped by some local minima; hence, these techniques may yield some inaccurate solutions, and they hardly converge to the real solution regardless \cite{mascarenhas2004bfgs}.

In this paper, thus, we combine two different approaches: namely, we use the neural network to discretize differential equations while we solve a system of nonlinear equations instead of the optimization problem. More specifically, in order to solve the overdetermined system of nonlinear equations that a fully connected neural network discretization yields, we introduce a randomized Newton's method, which randomly chooses equations from the overdetermined system and solves the nonlinear system adaptively. The remainder of the paper is structured as follows: In Section 2, we show the problem setup,  address the infeasibility of the traditional collocation method, and introduce the overdetermined nonlinear system that results from neural network discretization. Then, in Section 3, we present the detailed algorithm and the convergence analysis of the randomized Newton's method. Section 4 presents several numerical examples, ranging from linear to nonlinear differential equations on both one- and high- dimensional cases, to show the efficiency and feasibility of the randomized Newton's method. Moreover, we demonstrate with an application to pattern formation how this approach can be used to "learn" multiple solutions by coupling the randomized Newton's method with neural network discretizations.

\section{The problem setup}
We consider the following Laplace's equation
	\begin{equation}
	\label{numerical_implement}
	\left\{
		     \begin{array}{lr}
		     - \Delta u = f(u) & \text{in } \Omega \\
		     u = u_0 & \text{on } \partial \Omega, \\
		     \end{array}
	\right.
	\end{equation}
	where $\Omega \subset \mathbb{R}^d$ and $\partial \Omega$ is the boundary of the domain
	$\Omega$.  We have known that any critical point of the following energy functional is the solution of the Laplace's equation (\ref{numerical_implement}) \cite{evans2010partial}
	\begin{equation}\label{opt}
		E(u)= \int_\Omega \frac{1}{2}|\nabla u|^2 - F(u) dx,
	\end{equation}
	where $F'(x) = f(x)$ and $u$ belongs to the admissible set
	\begin{equation}
		\mathcal{A} := \{u\in C^2(\bar{\Omega}) | u = u_0 \text{ on } \partial \Omega\}.
	\end{equation}
	We apply an (n+1)-layer neural network $U(x;\theta)$ to approximate the solution
	to system (\ref{numerical_implement}), $u(x)$, namely,
	\begin{equation}
		\label{abstract_network}
		U(x;\theta) = W_n\sigma(W_{n-1}\cdots\sigma(W_2 \sigma(W_1 x + b_1) + b_2)\cdots + b_{n-1}) + b_n,
	\end{equation}
	where $\{W_i\}_{i=1}^{n}$ and $\{b_i\}_{i=1}^n$ are the weights and
	bias of the network, respectively, and 	$\sigma$ is the activation function such as the sin function, sigmoid function $\frac{e^x}{1+e^x}$, or ReLU $\max(x,0)$ \cite{sirignano2018dgm}.  For simplicity, we denote the set of all parameters as $\theta =
	\{W_1,\cdots,W_n,b_1,\cdots,b_n\}$ and 	the number of all parameters as $|\theta|$.
 Then, optimization techniques are successfully used to solve the resulting minimization problem based on (\ref{opt})
	\begin{equation}
		\min E(\theta) = \int_\Omega \frac{1}{2}|\nabla U(x;\theta) |^2 - F(U(x;\theta) ) dx +
				\int_{\partial\Omega} |U(x;\theta) - u_0(x)|^2 dS,
	\end{equation}\label{opt1}
where the second term represents the boundary conditions, which can be Dirichlet or Neumann \cite{PhysRevD.100.016002}. The numerical challenges of solving the optimization problem (\ref{opt1}) is that 1) the computational cost of function evaluations can be very large for high dimensional cases \cite{sirignano2018dgm}, and 2) the solutions are more likely to be trapped by some local minima, since the objective function $E(\theta)$ is usually highly non-convex \cite{weickert2001theoretical}.
In order to avoid these numerical difficulties, we solve the equation (\ref{numerical_implement}) directly by using the discretization of (\ref{abstract_network}).
Therefore, we get the following system of nonlinear equations:
	\begin{equation}
		\label{overdetermined}
	\mathbf{F}(\theta) =
		\begin{cases}
			\Delta U(x_i;\theta) + f(U(x_i;\theta))=0& \quad i=1,\cdots,N\\
			U(x_j;\theta) - u_0(x_j)=0& \quad j=1,\cdots, M
		\end{cases}
	\end{equation}
	where $\mathbf{F}: \mathbb{R}^{|\theta|} \rightarrow \mathbb{R}^{N+M}$, $x_{i}$ and $x_j$ are sample points on $\Omega$ and $\partial \Omega$ respectively.
The collocation method \cite{russell1972collocation} has been normally used to solve the resulting nonlinear system (\ref{overdetermined}) by taking $N+M=|\theta|$, that is, the number of sample points chosen is the same as the number of variables. However, the collocation method can not be used to solve the nonlinear system arising from neural network discretization due to its highly nonlinearity. We will use a simple example to illustrate this reason by considering
	\begin{equation}
	\label{collocation}
	\left\{
		     \begin{array}{lr}
		     u_{xx} = -4\pi^2\sin(2\pi x) & \text{on } (0,1), \\
		     u(0) = 0 \hbox{~and~} 		     u(1) = 0.  \\
		     \end{array}
	\right.
	\end{equation}
	We apply a one-hidden-layer neural network discretization, namely,
	\begin{equation}
		U(x; \theta) = W_2\sigma(W_1x + b_1) + b_2,
	\end{equation}
	where $\theta=\{W_1,W_2, b_1,b_2\}\in R^4$. Here we choose the activation function $\sigma(x)$ simply as $\sin(x)$, then $\theta=\{2\pi,1,0,0\}$ is the real solution. Since the number of parameters $|\theta|$ is 4, we use the collocation method to sample 2 points, $\{x_1,x_2\}$, in the domain $(0,1)$ and 2 points on
	the boundary. Thus the discretization system becomes:
	\begin{equation}\label{nonlinear}
		\mathbf{F}(\theta) =
		\left(
		\begin{array}{cr}
			\mathbf{F}_1(\theta)\\
			\mathbf{F}_2(\theta) \\
			\mathbf{F}_3(\theta) \\
			\mathbf{F}_4(\theta)
		\end{array}
		\right)
		=
		\left(
		\begin{array}{cr}
			-W_1^2W_2 \sin(W_1x_1 + b_1) + 4\pi^2\sin(2\pi x_1) \\
			-W_1^2W_2 \sin(W_1x_2 + b_1) + 4\pi^2\sin(2\pi x_2) \\
			W_2\sin(b_1) + b_2 \\
			W_2\sin(W_1 + b_1) + b_2
		\end{array}
		\right)
		=0.
	\end{equation}
By solving $(\mathbf{F}_1(\theta),\mathbf{F}_3(\theta))=(0,0)$ for $W_2$ and $b_2$ in term of $W_1$ and $b_1$, we have
	\begin{equation}
		\begin{cases}
			&W_2 = \frac{4\pi^2\sin(2\pi x_1)}{W_1^2\sin(W_1x_1 + b_1)}, \\
			&b_2 = -W_2\sin(b_1) = -\frac{4\pi^2\sin(2\pi x_1)}{W_1^2\sin(W_1x_1 + b_1)} \sin(b_1).
    		\end{cases}
	\end{equation}
Therefore, a simplified system of (\ref{nonlinear}) is written as
	\begin{equation}
		\mathbf{F}(W_1,b_1) =
		\begin{cases}
			\mathbf{F}_2(W_1,b_1) = -\frac{4\pi^2\sin(2\pi x_1)\sin(W_1x_2 + b_1)}{\sin(W_1x_1 + b_1)} + 4\pi^2\sin(2\pi x_2)\\
			\mathbf{F}_4(W_1,b_1) = \frac{4\pi^2\sin(2\pi x_1)\sin(W_1 + b_1)}{W_1^2\sin(W_1x_1 + b_1)} -\frac{4\pi^2\sin(2\pi x_1)}{W_1^2\sin(W_1x_1 + b_1)} \sin(b_1)
		\end{cases}\label{collocation_eq}.
	\end{equation}
We chose three groups of collocation points $\{x_1,x_2\}$:
\[CL_1=\{0.1,0.8\},~CL_2=\{0.8,0.9\},~CL_3=\{0.1,0.2\}.\]
Then we employed Newton's method to solve (\ref{collocation_eq}) with an initial guess $(W_1^0=1,~ b_1^0=1)$.
The solutions of nonlinear systems with three groups of collocation points $CL_i$ (i=1,2,3) are shown in Fig.  \ref{collocation_pic} (upper left): Newton's method finds the real solution for the nonlinear system with collocation points $CL_3$ while it delivers ``fake solutions" for other two systems. The reason is that there might be multiple solutions of (\ref{collocation_eq}) for any given sample points, although all the systems share one solution which corresponds to the real solution $u(x)$. Fig. \ref{collocation_pic} shows the multiple roots for different systems with three groups of collocation points.

	\begin{figure}[h]
		\centering
		\begin{subfigure}{0.45\textwidth} 
			\includegraphics[width=\textwidth]{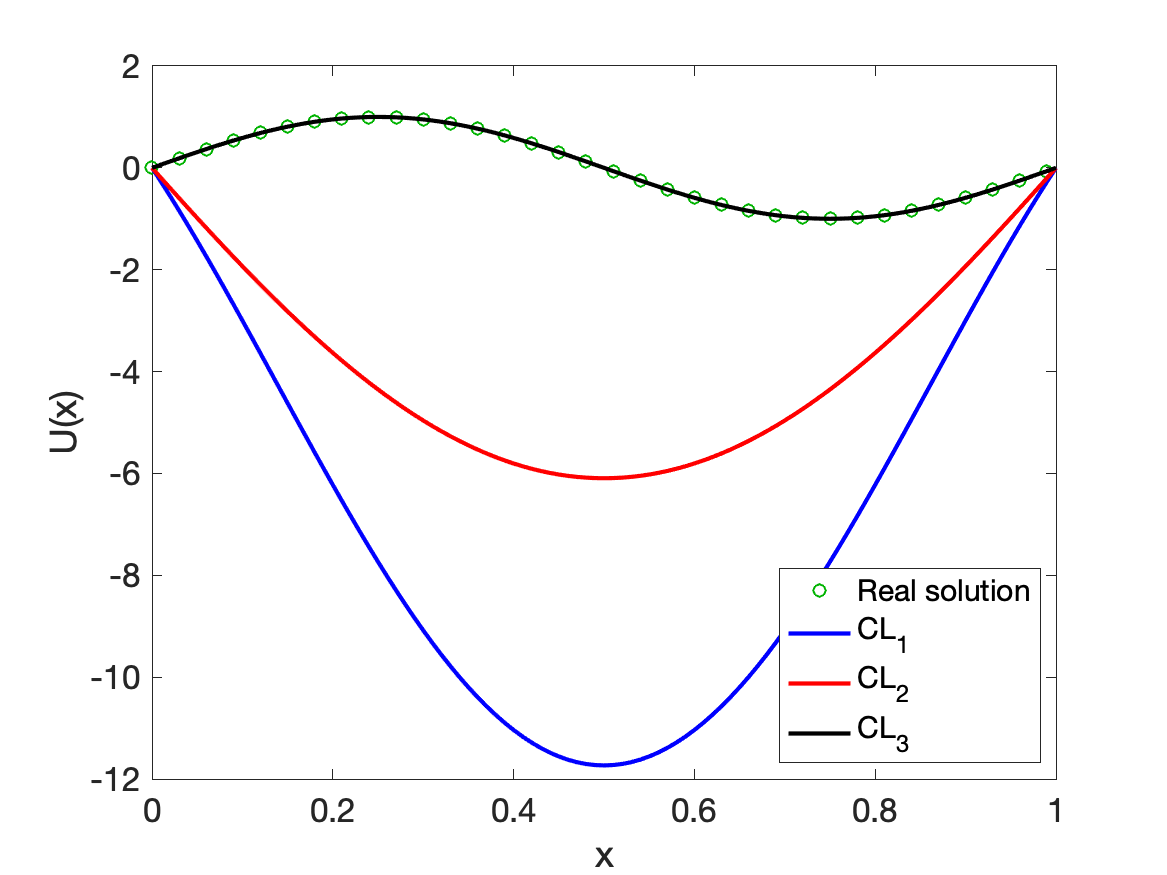}
			\label{collocation_pic}
		\end{subfigure}
		\vspace{1em}
		\begin{subfigure}{0.45\textwidth} 
			\includegraphics[width=\textwidth]{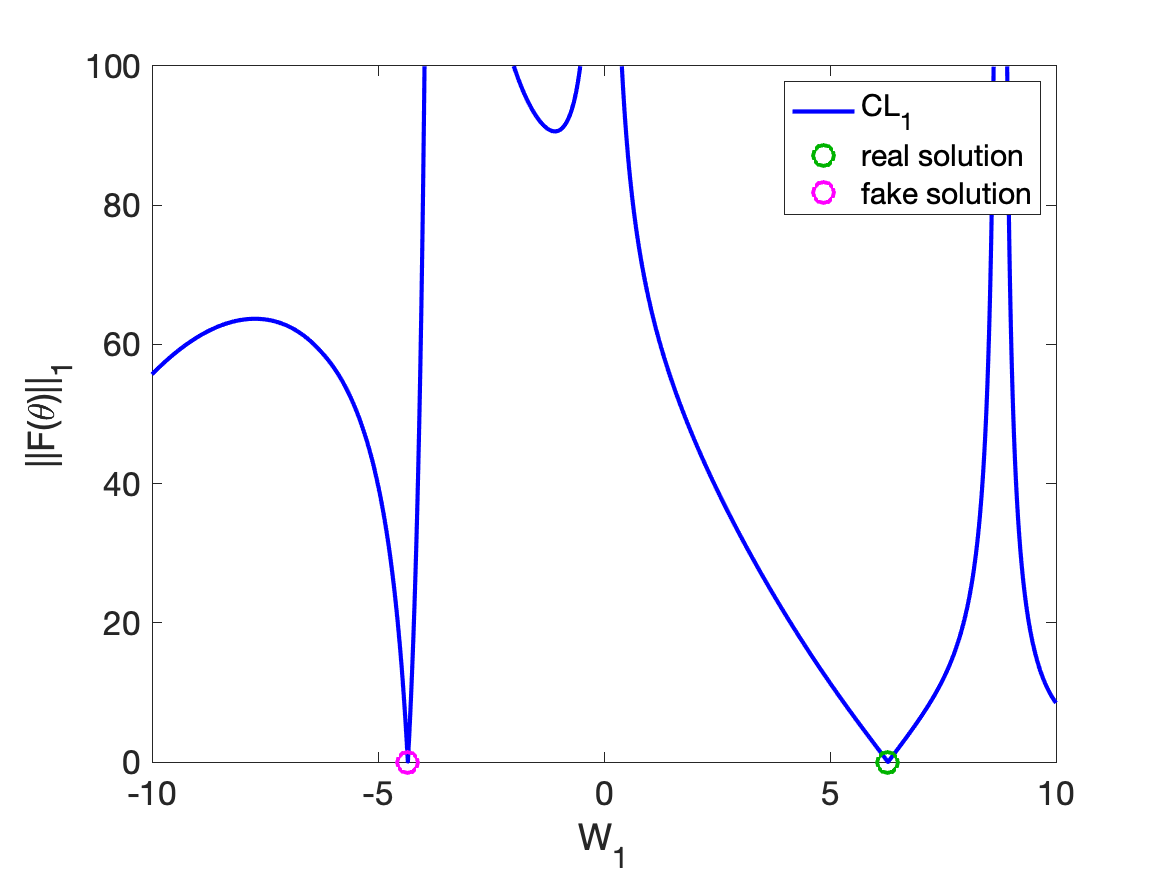}
			\label{CL_1}
		\end{subfigure}
		\vspace{1em} 
		\begin{subfigure}{0.45\textwidth} 
			\includegraphics[width=\textwidth]{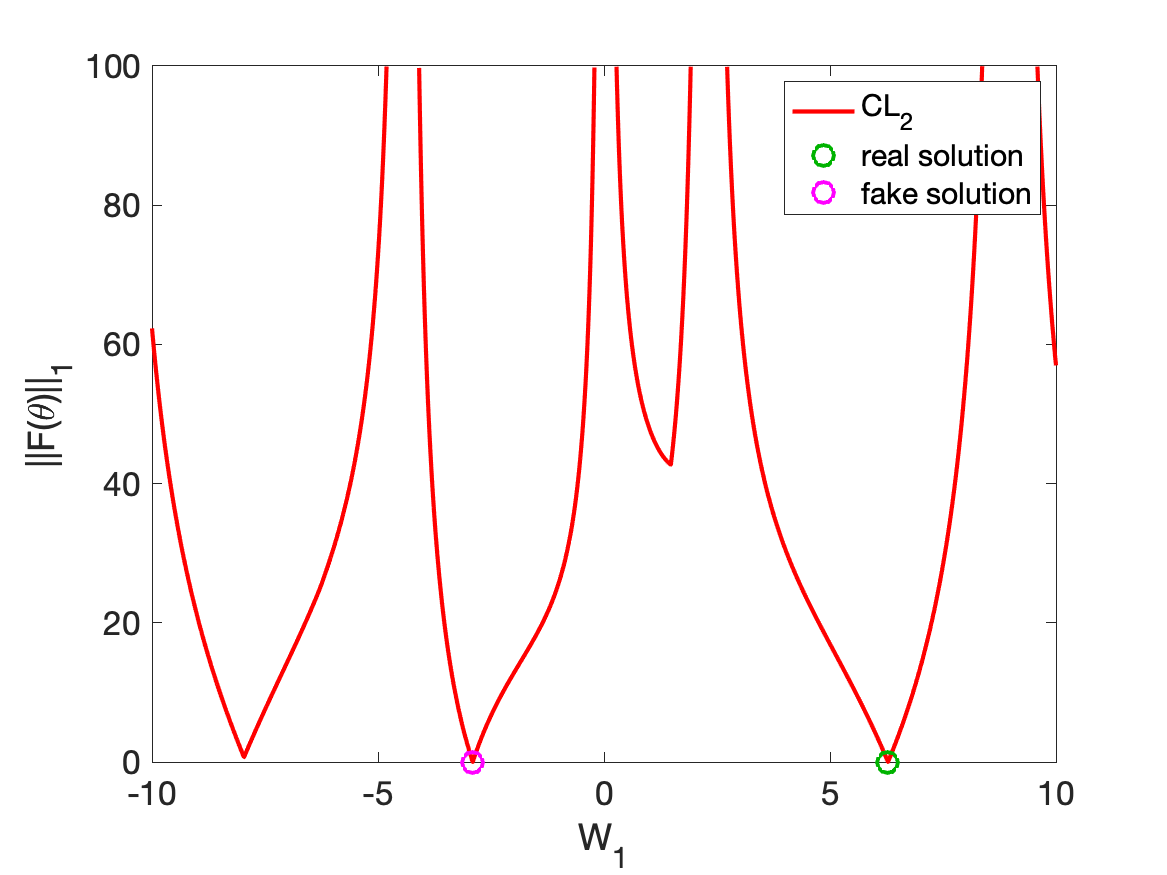}
			\label{CL_2}
		\end{subfigure}
		\vspace{1em}
		\begin{subfigure}{0.45\textwidth} 
			\includegraphics[width=\textwidth]{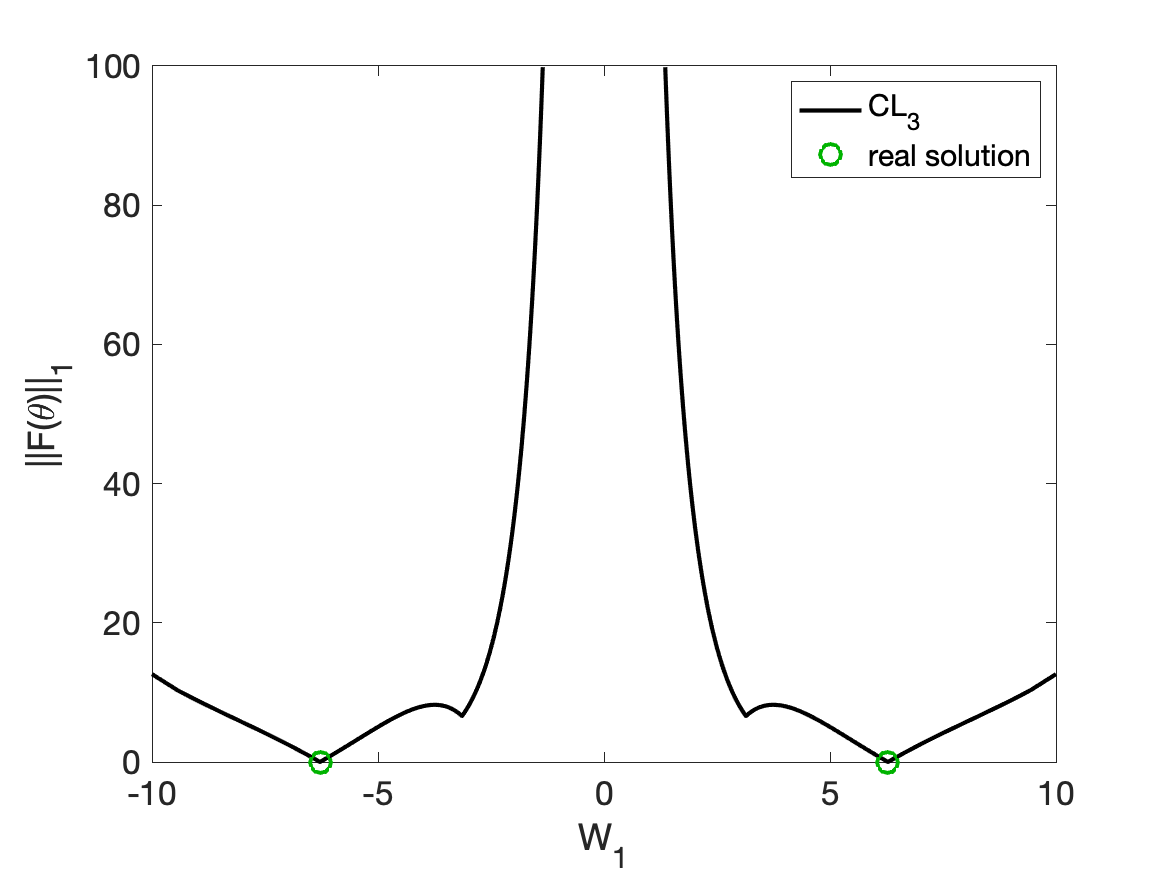}
			\label{CL_3}
		\end{subfigure}

		\caption{{\bf Upper left} Numerical Solutions of the collocation method for three groups of
		sample points $CL_i$ (i=1,2,3); {\bf Others:} $\|\mathbf{F}(\theta)\|_1$ v.s. $W_1$ for three groups of sample points $CL_i$ (i=1,2,3). Obviously, there are multiple roots for each nonlinear system that shares the same root $2\pi$. (We project $b_1$ to $W_1$ by using $b_1 = \frac{\tilde{b}_1}{\tilde{W}_1 - 2\pi}(W_1 - 2\pi)$, where $\tilde{b}_1$ and $\tilde{W}_1$ the solution other than $(W_1=2\pi \hbox{~and~} b_1=0)$.)}
		\label{collocation_pic}
	\end{figure}

Therefore, we need to sample many more points than the number of variables, namely, $N+M>|\theta|$, so that the system (\ref{overdetermined}) does not contain "fake solutions" with probability one \cite{wampler2011numerical, sommese2005introduction, leykin2011numerical}. Thus the system (\ref{overdetermined}) becomes an overdetermined system, since we need many more equations than variables. In this paper,  we developed an efficient randomized Newton's method to solve the overdetermined nonlinear system arising from the neural network discretization of differential equations.

\section{Randomized Newton's method}
We write the overdetermined system of nonlinear equations in the following general form
	\begin{equation}
		\mathbf{F}(\theta)=	\left(
			\begin{aligned}
				f_1&(\theta)\\
				f_2&(\theta)\\
				&\vdots\\
				f_n&(\theta)\\
			\end{aligned}\right)=\mathbf{0}, \label{Problem}
	\end{equation}
	where $\theta\in \mathbb{R}^m$ and $\mathbf{F}: \mathbb{R}^m \rightarrow
	\mathbb{R}^n$ $(n \gg m)$. In numerical algebraic geometry, the overdetermined
	system shown in (\ref{Problem}) can be solved by converting to a square system via solving
	$\mathbf{G}(\theta)=A_{m\times n} \mathbf{F}$, where $A_{m\times n}$ is an
	randomized matrix \cite{della1975numerical}. The drawback of this method is that 1) even the
	evaluation of this augmented system $\mathbf{G}$ could be problematic for
	large-scale systems, and 2) the number of sample points might be adaptive, so that constructing the  random matrix $A_{m\times n}$ each time would be time-consuming.

\subsection{Algorithm}
 In this paper, we develop a randomized Newton's method, namely,
	\begin{equation}
	\label{scheme_1}\theta^{k+1}=\theta^{k}-\nabla \tilde{\mathbf{F}}^{-1}(\theta^{k}) \
	\tilde{\mathbf{F}}(\theta^{k}),
	\end{equation}
	where $\tilde{\mathbf{F}}$ is a system with randomly chosen $m$ equations from
	$\mathbf{F}$, $\nabla \tilde{\mathbf{F}}$ is the Jacobian matrix
	of $\tilde{\mathbf{F}}$ with respect to $\theta$. If
	$\nabla\tilde{\mathbf{F}}(\theta^{k})$ is singular, then the Gauss-Newton's
	method will be employed. The randomized Newton's method can be rewritten with a random variable $\xi$ on a probability space $(\Omega, \mathcal{F}, \mathcal{P})$:
	\begin{equation}
		\xi: \Omega \rightarrow \Gamma,
	\end{equation}
	where $\Gamma$ is a set with all the combinations of $m$ numbers out of $\{1,2,\dots,n\}$. 
	Since $|\Gamma|=  {{n}\choose{m}}$, we denote
	\begin{equation}
		\Gamma = \{\gamma_1, \gamma_2, \dots, \gamma_{{n}\choose{m}}\},
	\end{equation}
	and assume random variable $\xi$ follows the uniform distribution, namely,
	$\mathcal{P}(\xi = \gamma_i)=\frac{1}{{{n}\choose{m}}}$ for $1 \leq i \leq {{n}\choose{m}}$.
	Then we rewrite $\tilde{\mathbf{F}}$ as
	\begin{equation}
	\tilde{\mathbf{F}}=	\mathbf{F}(\theta,\gamma_s):=	
		\left(
			\begin{aligned}
				f_{s_1}&(\theta)\\
				f_{s_2}&(\theta)\\
				&\vdots\\
				f_{s_m}&(\theta)\\
			\end{aligned}
		\right),
	\end{equation}
	where $\gamma_s=\{s_1,\dots,s_m\}$ and $\{s_1, \dots, s_m\}\subset \{1,\dots,n\}$ Similarly, the randomized Newton's method is rewritten as \begin{equation}
		\label{Rand_newton}
		\theta^{k+1} = \theta^k - \nabla \mathbf{F}^{-1}(\theta^k, \xi_k)\cdot \mathbf{F}(\theta^k, \xi_k),
	\end{equation}
which has  a more general form as follows
\begin{equation}
	\label{rnm_sde}
	\theta^{k+1} = \theta^{k} - \eta \nabla\mathbf{F}^\dag (\theta^k)\mathbf{F}(\theta^k) + \sqrt{\eta} \mathcal{R}(\theta^k,\xi_k),
\end{equation}
where $\eta$ is the step-length usually determined by trust region and line search \cite{nocedal1998combining}. In the algorithm (\ref{Rand_newton}), we choose $\eta=1$.
Here $\nabla\mathbf{F}^\dag (\theta^k)$ is the pseudoinverse of $\nabla\mathbf{F}(\theta^k)$ and \[\mathcal{R}(\theta^k,\xi_k) = \sqrt{\eta} (\nabla\mathbf{F}^\dag(\theta^k)\mathbf{F}(\theta^k) - \nabla\mathbf{F}^{-1}(\theta^{k},\xi_k) \mathbf{F}(\theta^{k},\xi_k)).\]
By assuming
\begin{equation}
	\label{expect_assump}
	\nabla\mathbf{F}^\dag(\theta)\mathbf{F}(\theta) = \frac{1}{|\Gamma|} \sum_{i=1}^{|\Gamma|}\nabla\mathbf{F}^{-1}(\theta,i) \mathbf{F}(\theta,i),
\end{equation}
we have
$	\mathbb{E}[\mathcal{R}(\theta,\xi)] = 0.$
Then the covariance matrix of $\mathcal{R}(\theta,\xi)$ is
\begin{equation}
	\label{cov}
	\begin{aligned}
	\Sigma(\theta) &= \mathbb{E}[\mathcal{R}(\theta,\xi)\mathcal{R}^T(\theta,\xi)]\\
	&= \eta\mathbb{E}[(\nabla\mathbf{F}^\dag(\theta)\mathbf{F}(\theta) - \nabla\mathbf{F}^{-1}(\theta,\xi) \mathbf{F}(\theta,\xi))(\nabla\mathbf{F}^\dag(\theta)\mathbf{F}(\theta) - \nabla\mathbf{F}^{-1}(\theta,\xi) \mathbf{F}(\theta,\xi))^T]\\
	&= \eta\Big(\mathbb{E}\big[\nabla\mathbf{F}^{-1}(\theta,\xi) \mathbf{F}(\theta,\xi)(\nabla\mathbf{F}^{-1}(\theta,\xi) \mathbf{F}(\theta,\xi))^T\big] - \nabla\mathbf{F}^\dag(\theta)\mathbf{F}(\theta)(\nabla\mathbf{F}^\dag(\theta)\mathbf{F}(\theta))^T\Big)\\
	&= \frac{\eta}{|\Gamma|}\sum_{i=1}^{|\Gamma|}[\nabla\mathbf{F}^{-1}(\theta,i) \mathbf{F}(\theta,i)(\nabla\mathbf{F}^{-1}(\theta,i) \mathbf{F}(\theta,i))^T] - \eta \nabla\mathbf{F}^\dag(\theta)\mathbf{F}(\theta)(\nabla\mathbf{F}^\dag(\theta)\mathbf{F}(\theta))^T.
	\end{aligned}
\end{equation}
{\bf Remark:} We will look at the randomized Newton's method from the  stochastic differential equation (SDE) point of view and consider the following general form of SDE:
\begin{equation}
	\label{sde}
	d\theta_t = b(\theta_t)dt + \sigma(\theta_t)dW_t, \quad \theta_0 = \theta_{init}.
\end{equation}
Then the Euler-Maruyama discretization \cite{kloeden2013numerical} of (\ref{sde}) becomes
\begin{equation}
	\label{sde_dis}
	\theta^{k+1} = \theta^k + b(\theta^k)\Delta t + \sqrt{\Delta t} \sigma(\theta^k)Z^k
\end{equation}
where $Z^k \sim \mathcal{N}(0,I)$. If we choose $\Delta t= \eta$, $b(\cdot) = - \nabla\mathbf{F}^\dag (\cdot)\mathbf{F}(\cdot)$, and $\sigma(\cdot) = (\Sigma(\cdot))^{\frac{1}{2}}$, then Eq. (\ref{sde_dis}) becomes
\begin{equation}
	\label{sde_dis_approx}
	\theta^{k+1} = \theta^k - \eta  \nabla\mathbf{F}^\dag (\theta^k)\mathbf{F}(\theta^k)+ \sqrt{\eta } (\Sigma(\theta^k))^{\frac{1}{2}} Z^k,
\end{equation}
Thus the SDE (\ref{sde}) is an approximation of (\ref{rnm_sde}) in the weak sense introduced in \cite{li2017stochastic}.

\subsection{Convergence analysis}
Next we define tensor $\nabla^2 \mathbf{F}(\theta, \xi)$ as follows,
	\begin{equation}
		[\nabla^2 \mathbf{F}(\theta,\xi)]_{ijk} := [\nabla^2 f_i(\theta,\xi)]_{jk}, \quad i,j,k \in \{1,2,\cdots,m\}
	\end{equation}
where $\nabla^2 f_i(\theta,\xi)$ is the Hessian matrix of $f_i(\theta,\xi)$. Accordingly, we define the multiplication of the tensor with vectors as, for $\forall a,b\in R^m$,
	\begin{equation}
		[a^T \cdot \nabla^2 \mathbf{F}(\theta,\xi)\cdot b]_i := \sum_{k=1}^m\sum_{j=1}^m a_j [\nabla^2 f_i(\theta,\xi)]_{jk} b_k, \quad i \in \{1,2,\cdots,m\}.
	\end{equation}
	Then $\|\nabla^2 \mathbf{F}(\theta,\xi)\|$ is defined as $\|\nabla^2 \mathbf{F}(\theta,\xi)\| = \max_{i\in\{1,\cdots,m\}}\|\nabla^2 f_i(\theta,\xi)\|$. Thus we summarize the local convergence of the randomized Newton's method in the following theorem.

	\begin{thm}\label{thm1}
		Suppose $\theta^*$ is the solution to $\mathbf{F}(\theta)=0$. Assuming
		that $\nabla A\mathbf{F}(\cdot)$ is invertible and continuous,
		and that $\nabla^2 A\mathbf{F}(\cdot)$  is continuous in a
		small neighborhood of $\theta^*$ for any permutation matrix
		$A\in R^{m\times n}$. Then for scheme (\ref{Rand_newton}), we have \[\mathbb{E}(\|\theta^* -
		\theta^k\|) \le C\Big(\frac{1}{2}\Big)^{2^k-1},\] which implies  the
		quadratic convergence.
	\end{thm}
	
	\begin{proof}
		We consider the  following Taylor expansion of $\mathbf{F}(\theta, \xi_k)$ at $\theta^k$ for $\theta^*$:
		\begin{equation}
			\label{eqn1.7}
			\mathbf{F}(\theta^*, \xi_k) = \mathbf{F}(\theta^k, \xi_k) + \nabla \mathbf{F}(\theta^k, \xi_k)\cdot(\theta^* - \theta^k) + \frac{1}{2}(\theta^* - \theta^k)^T\cdot \nabla^2 \mathbf{F}(\mathbf{t}_k, \xi_k)\cdot (\theta^* - \theta^k),
		\end{equation}
		where $\mathbf{t}_k$ is between $\theta^*$ and $\theta^k$.
		Since $\theta^*$ is a solution to $\mathbf{F}(\theta)=0$,  and is also the solution to $\mathbf{F}(\theta, \xi_k)=0$. Then (\ref{eqn1.7}) becomes
		\begin{equation}
			\mathbf{0}=\mathbf{F}(\theta^*, \xi_k) = \mathbf{F}(\theta^k, \xi_k) + \nabla \mathbf{F}(\theta^k, \xi_k)\cdot(\theta^* - \theta^k) + \frac{1}{2}(\theta^* - \theta^k)^T\cdot \nabla^2 \mathbf{F}(\mathbf{t}_k, \xi_k)\cdot (\theta^* - \theta^k).\label{eqn1.8}
		\end{equation}
		By multiplying $\nabla \mathbf{F}^{-1}(\theta^k, \xi_k)$ on both sides of (\ref{eqn1.8}), we get
		\begin{equation}
			\label{eqn1.9}
			\nabla \mathbf{F}^{-1}(\theta^k, \xi_k)\cdot \mathbf{F}(\theta^k, \xi_k) + (\theta^* - \theta^k) = -\frac{1}{2}\nabla \mathbf{F}^{-1}(\theta^k, \xi_k)\cdot (\theta^* - \theta^k)^T \nabla^2 \mathbf{F}(\mathbf{t}_k, \xi_k) (\theta^* - \theta^k).
		\end{equation}
		By substituting (\ref{Rand_newton}) into the left-hand side of (\ref{eqn1.9}), we have
		\begin{equation}
		\label{eqn1.10}
			\theta^* - \theta^{k+1} =  -\frac{1}{2}\nabla \mathbf{F}^{-1}(\theta^k, \xi_k)\cdot (\theta^* - \theta^k)^T \nabla^2 \mathbf{F}(\mathbf{t}_k, \xi_k) (\theta^* - \theta^k).
		\end{equation}
		By taking the expectation on both sides of (\ref{eqn1.10}), we obtain
		\begin{equation}
			\label{eqn1.11}
			\begin{aligned}
				\mathbb{E}(\|\theta^* - \theta^{k+1}\|) &= \frac{1}{2}\mathbb{E}\big(\|\nabla \mathbf{F}^{-1}(\theta^k, \xi_k)\cdot (\theta^* - \theta^k)^T \nabla^2 \mathbf{F}(t_k, \xi_k) (\theta^* - \theta^k)\|\big) \\
				&= \frac{1}{2}\mathbb{E}_{\xi_0 \xi_1 \dots \xi_{k-1}}\big(\mathbb{E}_{\xi_k}[\|\nabla \mathbf{F}^{-1}(\theta^k, \xi_k)\cdot (\theta^* - \theta^k)^T \nabla^2 \mathbf{F}(t_k, \xi_k) (\theta^* - \theta^k)\|]\big) \\
				&= \frac{1}{2}\mathbb{E}\big(\frac{1}{|\Gamma|}\cdot \sum_{i=1}^{|\Gamma|}\|\nabla  \mathbf{F}^{-1}( \theta^k, \gamma_i)\cdot ( \theta^* -  \theta^k)^T \nabla^2 \mathbf{F}( \mathbf{t}_k, \gamma_i) ( \theta^* -  \theta^k)\|\big) \\
				&\le \frac{1}{2}\frac{1}{|\Gamma|}\cdot \sum_{i=1}^{|\Gamma|}\mathbb{E}\big(\|\nabla  \mathbf{F}^{-1}( \theta^k, \gamma_i)\|\cdot \|\nabla^2  \mathbf{F}( \mathbf{t}_k, \gamma_i)\|\cdot \| \theta^* -  \theta^k\|^2\big) \\
			\end{aligned}
		\end{equation}

		Our assumptions in the theorem are equivalent to, for any
				$i\in \{1,2,\dots,|\Gamma|\}$,
		\begin{enumerate}
			\item $\nabla  \mathbf{F}( \theta^*, \gamma_i)$ is invertible and
				$\nabla  \mathbf{F}(\cdot, \gamma_i)$ is continuous;
			\item $\nabla^2  \mathbf{F}(\cdot, \gamma_i)$ is continuous.
		\end{enumerate}
		Therefore, $\exists M>0$  in a neighborhood of $\theta^*$ such that
			$\|\nabla  \mathbf{F}^{-1}( \theta, \gamma_i)\| \cdot \|\nabla^2
			\mathbf{F}( \mathbf{t}_k, \gamma_i)\|\le M$ for $\forall i\in \{1,2,\dots, |\Gamma|\}$.  Then
			(\ref{eqn1.11}) becomes
		\begin{equation}
			\label{eqn1.12}
			\mathbb{E}\| \theta^* - \theta^{k+1}\| \le
			\frac{M}{2}\mathbb{E}\|\theta^* - \theta^k\|^2.
		\end{equation}
		Similar to (\ref{eqn1.11}), we have
		\begin{equation}
		\mathbb{E}\|\theta^* - \theta^k\|^2	\leq \frac{M^2}{4}\mathbb{E}\|\theta^* - \theta^{k-1}\|^4,
		\end{equation}
which implies 
		\begin{equation}
			\mathbb{E}\| \theta^* - \theta^{k+1}\| \le
			\frac{M^3}{2^3}\mathbb{E}\|\theta^* - \theta^{k-1}\|^4\leq \cdots \leq 			\frac{M^{2^{k+1}-1}}{2^{2^{k+1}-1}}\mathbb{E}\|\theta^* - \theta^{0}\|^{2^{k+1}}.
		\end{equation}
		Assuming that $\|\theta^* - \theta^0\|\le r$ where $r\le \frac{1}{M}$, we have
		\begin{equation}
			\mathbb{E}\| \theta^* - \theta^{k+1}\| \le
			C\Big(\frac{1}{2} \Big)^{2^{k+1}-1},
		\end{equation}
		which implies the quadratic convergence.
	\end{proof}

\section{Numerical Experiments}
In this section, we demonstrate the feasibility of the randomized Newton's method on several examples and choose $\|\tilde{F}(\theta)\| < 5\times10^{-3}$ as the stopping criteria.

\subsection{1D Examples}
\subsubsection{An example with the analytical solution}
	First we show the feasibility of the randomized Newton's method on (\ref{collocation})
	which the traditional collocation method fails. One-hidden-layer neural networks with different numbers
	of nodes,  $U(x;\theta)$, are used to approximate the solution of (\ref{collocation}). By using three different uniform grids on $[0,1]$ with step size $0.1$, $0.02$, and $0.01$ (namely, $n= 11$, $51$ and $101$ respectively in (\ref{Problem})), the randomized Newton's method shows a good agreement with the real solution. For instance, Figure \ref{fig1} shows the numerical solution with 10 hidden nodes and $101$ sample points versus the real solution. More specifically, we list numerical errors between numerical solutions and the real solution with different numbers of nodes and sample points in Table \ref{table_1}. Here the numerical error $Err(U(x;\theta)-u(x))$ is defined as:
	 \begin{equation}
	 	Err(U(x;\theta)-u(x)) = \sqrt{\int_0^1|U(x;\theta) - u(x)|^2 dx}.
	 \end{equation}
From Table \ref{table_1}, we note that
1) a better approximation is achieved by  increasing nodes due to the universal approximation theory \cite{cybenko1989approximation}, and
2) the redundancy happens even in the one-hidden-layer neural network \cite{Medler1994, medler1994using}; therefore the randomized Newton's method has more
 iterations than expected.

	\begin{figure}[h]
		\centering
		\includegraphics[width=0.45\textwidth]{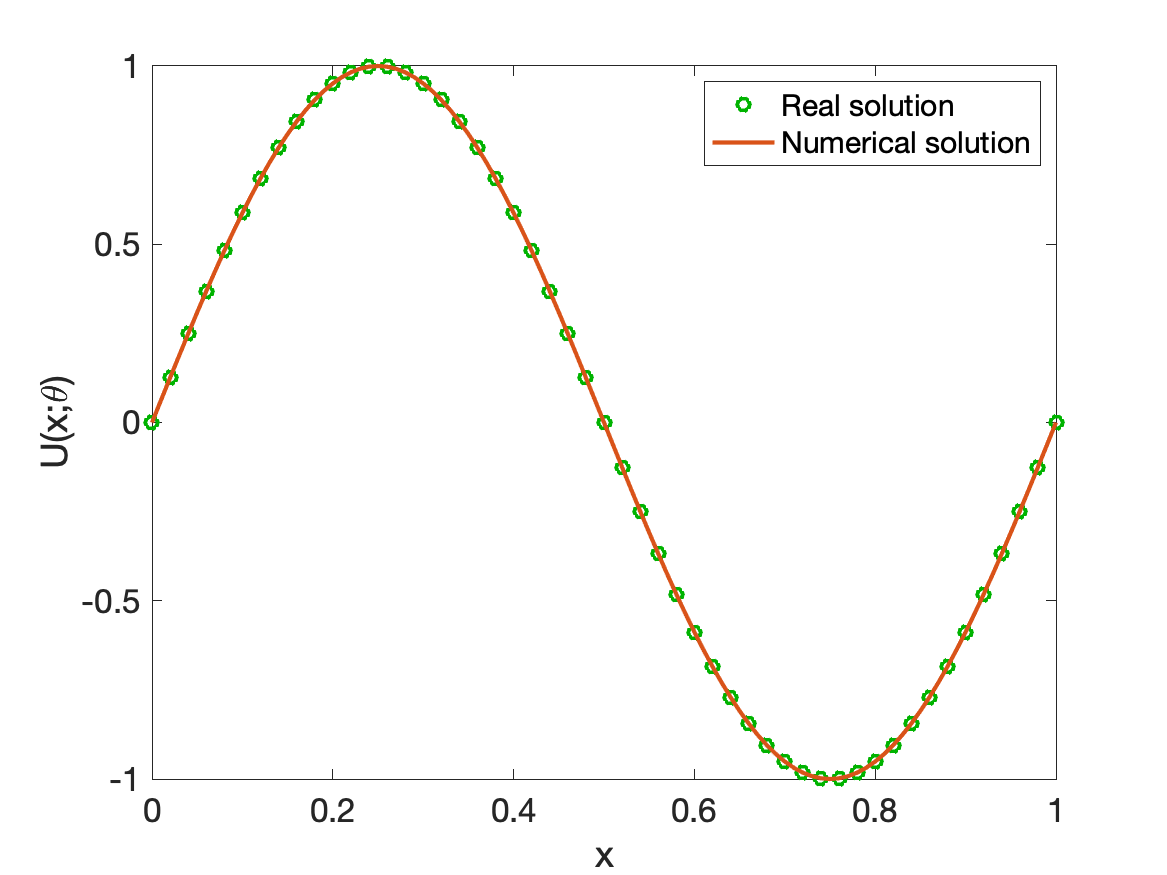}\label{example1}
		\caption{The numerical solution with 10 nodes and $101$ sample points V.S. the real solution
		to system (\ref{collocation}).}
		\label{fig1}
	\end{figure}

\begin{table}
\centering
\begin{tabular}{|c|c|c|c|c|}
\hline
$\#$ of sample points&\# of nodes&$\#$ of variables&$\#$ of iterations& Errors\\
\hline
\multirow{3}*{11}&1&4&18&4.1e-4\\
\cline{2-5}
~&2&7&24&1.2e-4\\
\cline{2-5}
~&3&10&33&1.5e-4\\
\hline
\multirow{3}*{51}&1&4&28&2.3e-4\\
\cline{2-5}
~&5&16&33&4.0e-4\\
\cline{2-5}
~&10&31&37&5.6e-5\\
\hline
\multirow{3}*{101}&1&4&22&2.4e-4\\
\cline{2-5}
~&5&16&64&8.0e-5\\
\cline{2-5}
~&10&31&50&6.0e-6\\
\hline
\end{tabular}
\caption{Numerical errors for different number of nodes and sample points.}
\label{table_1}
\end{table}

\subsubsection{An example with multiple solutions}
Secondly, we consider the following differential equation with multiple solutions:
	\begin{equation}
	\label{example_2}
	\left\{
		     \begin{array}{lr}
		     u_{xx} = f(u) \quad \text{ on } (0,1), \\
		     u'(0) = 0 \text{ and } u(1) = 0,
		     \end{array}
	\right.
	\end{equation}
	where
	\begin{equation}
		\label{example_2_f}
		f(u) = - \lambda (1+u^p), \quad p \in \mathbb{N} \text{ and } \lambda
		\ge 0.
	\end{equation}
	
When $p = 4$, we have known that there are two solutions if $\lambda< \lambda^*$ ($\lambda^* \approx 1.30107$), and one solution if $\lambda=\lambda^*$ \cite{hao2014bootstrapping,hu2011blow}.	The real solutions of (\ref{example_2}) can be determined by solving the nonlinear equation below for $u(0)$:
\begin{equation}
	\label{example_2_3}
G(u_0):=\int_0^{u_0}\frac{ds}{\sqrt{F(u_0) - F(s)}} - \sqrt{2} = 0,
\end{equation}
where $F(u) = \int_0^u\lambda(1+s^p)ds$ and $u_0 = u(0)$. Then we test the randomized Newton's method for both $\lambda=1.2$ and $\lambda=\lambda^*$ by employing
a one-hidden-layer neural network with two nodes and 101 uniform sample points on $[0,1]$. More specifically, the numerical solution is written as
	\begin{equation}
		U(x; \theta) = W_2\sigma(W_1x + b_1) + b_2,
	\end{equation}
	where $W_1,b_1\in \mathbb{R}^{2\times 1}$, $W_2\in \mathbb{R}^{1\times 2}$
	, $b_2\in \mathbb{R}$ and $\theta = \{W_1,b_1,W_2,b_2\}$. For $\lambda=1.2$, we
	choose two different initial values \[\theta_0^1 = \{(1,1)^T,(1,1)^T,(1,1),1\} \hbox{ and }
 \theta_0^2 = \{(5, 0.5)^T, (1, -3)^T, (1, -27), 2\}\] and obtain two numerical solutions, as shown in
Figure \ref{eg2_fig1}; For $\lambda = \lambda^*$, the two initial values yield the same numerical solution, as shown in
Figure \ref{eg2_fig2}. Moreover, we also test the algorithm on neural networks with different structures and show
numerical errors in Table \ref{table_2}.
	
\begin{table}
\centering
\begin{tabular}{|c|c|c|c|}
\hline
$\#$ of sample points&Width of hidden layers&$\#$ of variables& Errors\\
\hline
\multirow{4}*{101}&2&7&7.5e-2\\
\cline{2-4}
~&5&16&4.8e-2\\
\cline{2-4}
~&(2,2)&13&3.3e-3\\
\cline{2-4}
~&(3,2)&17&2.4e-3\\
\hline
\end{tabular}
\caption{Numerical errors for neural networks with different structures (the first two rows are neural networks with a 1-hidden layer while the last two rows are neural networks with 2-hidden layers).}
\label{table_2}
\end{table}

	\begin{figure}
		\centering
		\begin{subfigure}{0.45\textwidth} 
			\includegraphics[width=\textwidth]{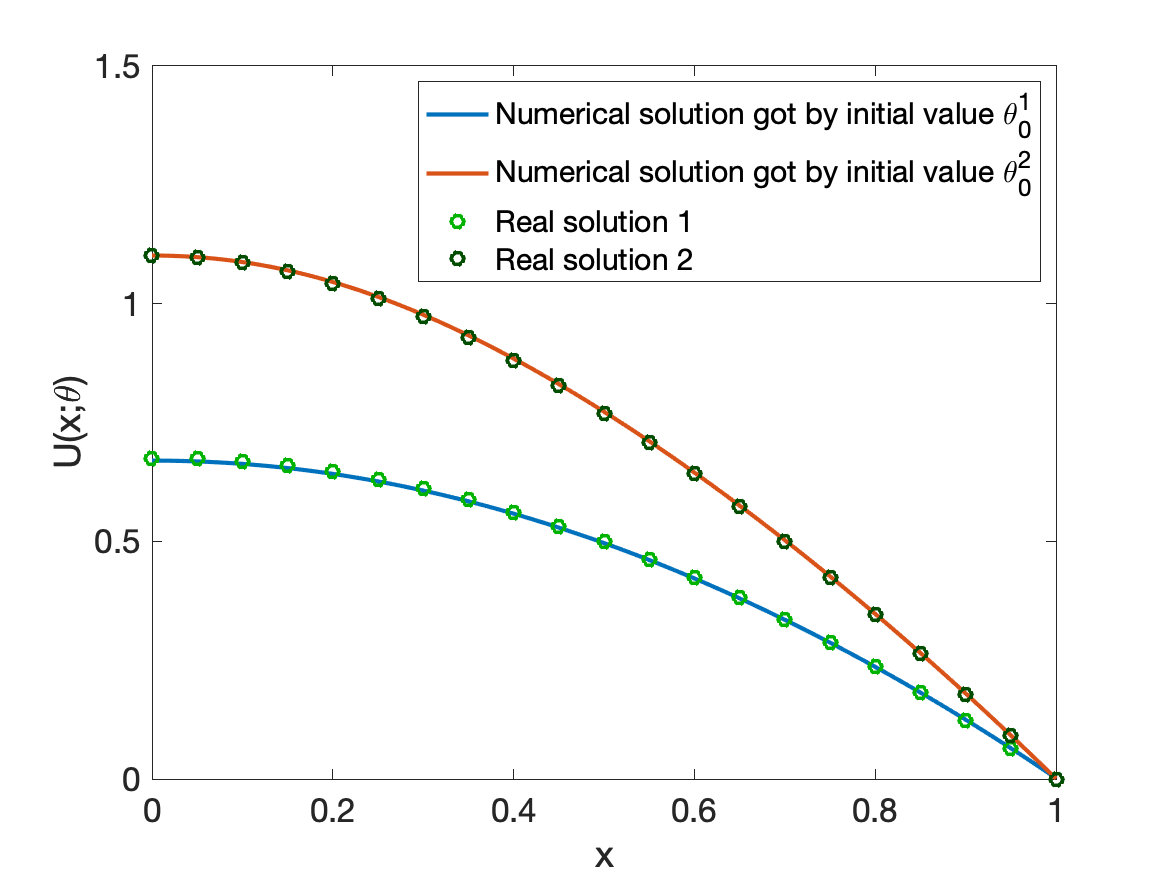}
			\caption{$\lambda=1.2$} 
			\label{eg2_fig1}
		\end{subfigure}
		\vspace{1em} 
		\begin{subfigure}{0.45\textwidth} 
			\includegraphics[width=\textwidth]{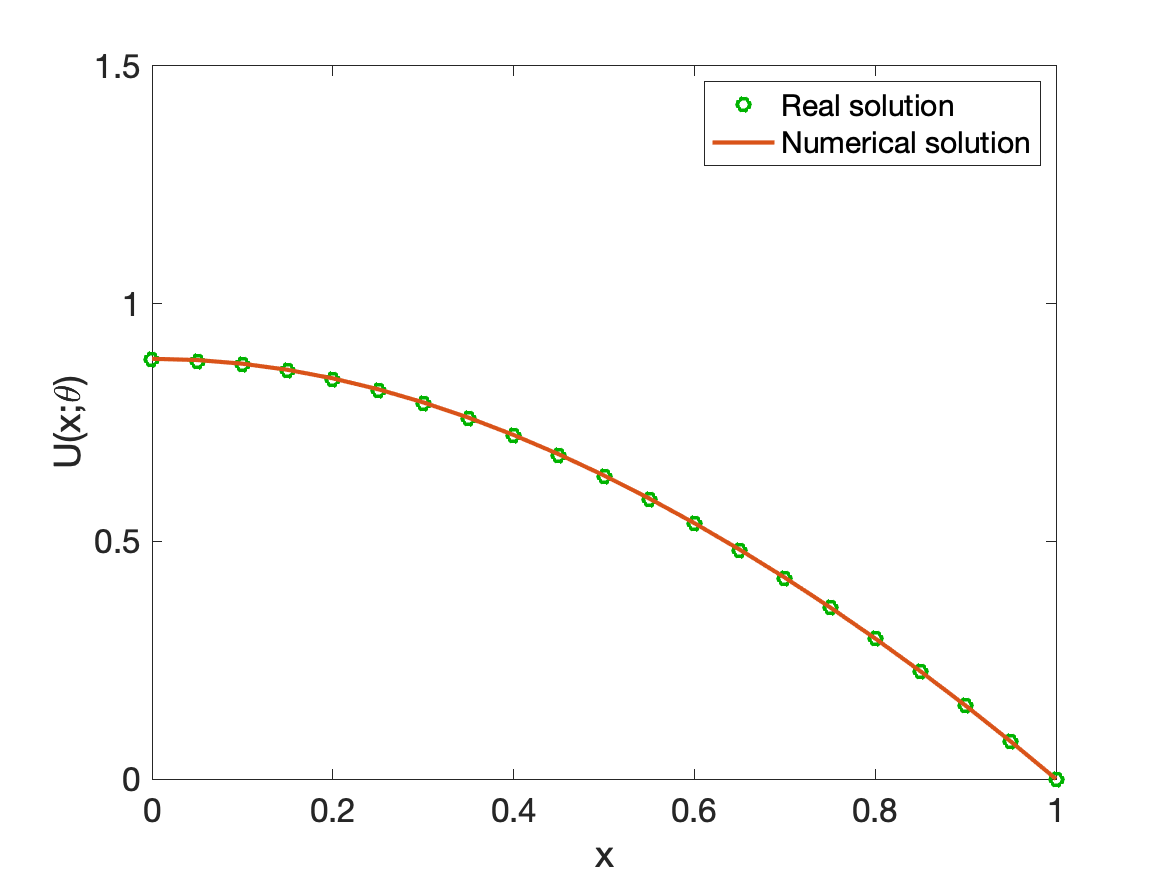}
			\caption{$\lambda=\lambda^*$} 
			\label{eg2_fig2}
		\end{subfigure}
		\caption{Numerical solutions V.S. real solutions of system (\ref{example_2}).} 
	\end{figure}

\subsubsection{The 1D Burger's equation}
Next we consider the 1D Burger's equation with a viscosity term:
	\begin{equation}
	\label{example_3}
	\left\{
		     \begin{array}{lr}
		     -\epsilon u_{xx} + (\frac{u^2}{2})_x = \sin(x)\cos(x) & \text{on } (0,\pi), \\
		     u(0) = 0 \text{ and } u(\pi) = 0, \\
		     \end{array}
	\right.
	\end{equation}
	where $\epsilon$ is the viscosity coefficient.
We use  a one-hidden-layer neural network  with ten nodes to approximate the solution of (\ref{example_3}) and  $101$ uniform sample points  on $[0,1]$.
When $\epsilon=1$, the solution is unique and converges to the entropy solution of the 1D Burger's equation as $\epsilon\rightarrow 0$ \cite{de2004minimal}. The analytical
solution of system (\ref{example_3}) when $\epsilon=0$ has the following form \cite{hao2013homotopy,chou2006high,chou2007high}
\begin{equation}
\label{entropy_solu}
	u(x) = \begin{cases}
		\sin(x) \quad &0\le x < x_0,\\
		-\sin(x) \quad &x_0 < x \le \pi,
	\end{cases}
\end{equation}
where $x_0\in [0,\pi]$ is the shock location. Specifically when $x_0 = \frac{\pi}{2}$, it becomes the entropy solution due to the symmetry.
Therefore, we use the randomized Newton's method to solve (\ref{example_3}) as tracking $\epsilon$ from $1$ to $0$.  In order to test the randomized Newton's method, we employ two tracking methods: 1) homotopy tracking \cite{hao2013homotopy}, namely,  using the previous solution as the initial guess; 2) the randomized initial guess for each $\epsilon$.
We list  numerical performance of two tracking methods based on the randomized Newton's method in Table \ref{table_3},
which shows numbers of iterations and condition numbers for each $\epsilon$. Obviously, homotopy tracking converges much faster and also captures the singularity at $\epsilon=0$ (see more details in \cite{hao2013homotopy}).  The numerical solutions are plotted in Figure \ref{eg4_fig1} and \ref{eg4_fig2}, which shows that the homotopy tracking obtains the entropy solution while the other method converges to an artificial steady state of $x_0=1$ when $\epsilon\rightarrow 0$. The numerical errors at $\epsilon=0$ are $3.6\times 10^{-3}$ for homomtopy tracking and $4.1\times 10^{-3}$ for the tracking with a random initial guess. This example demonstrates that the randomized Newton's method can be coupled with different tracking methods and computes different solutions.

	\begin{figure}
		\centering
		\begin{subfigure}{0.45\textwidth} 
			\includegraphics[width=\textwidth]{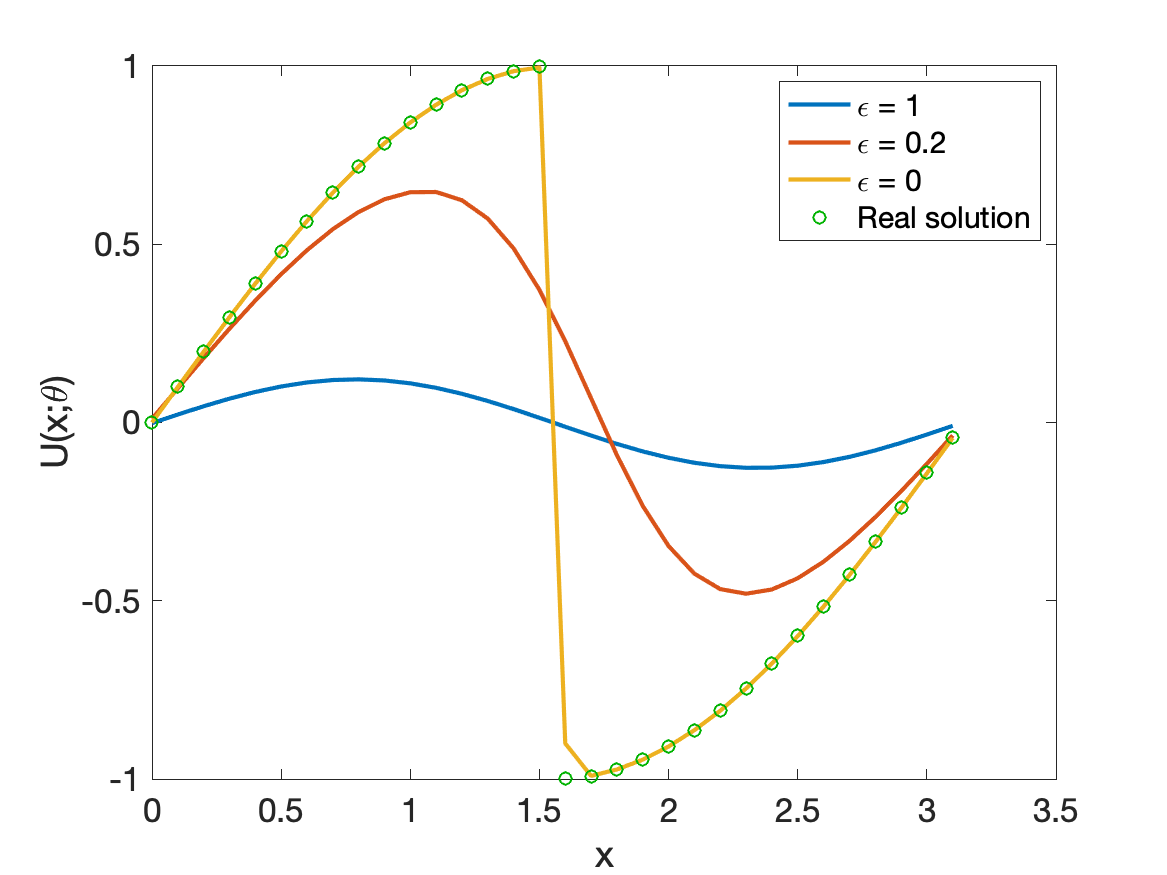}
			\caption{Homotopy tracking} 
			\label{eg4_fig1}
		\end{subfigure}
		\vspace{1em} 
		\begin{subfigure}{0.45\textwidth} 
			\includegraphics[width=\textwidth]{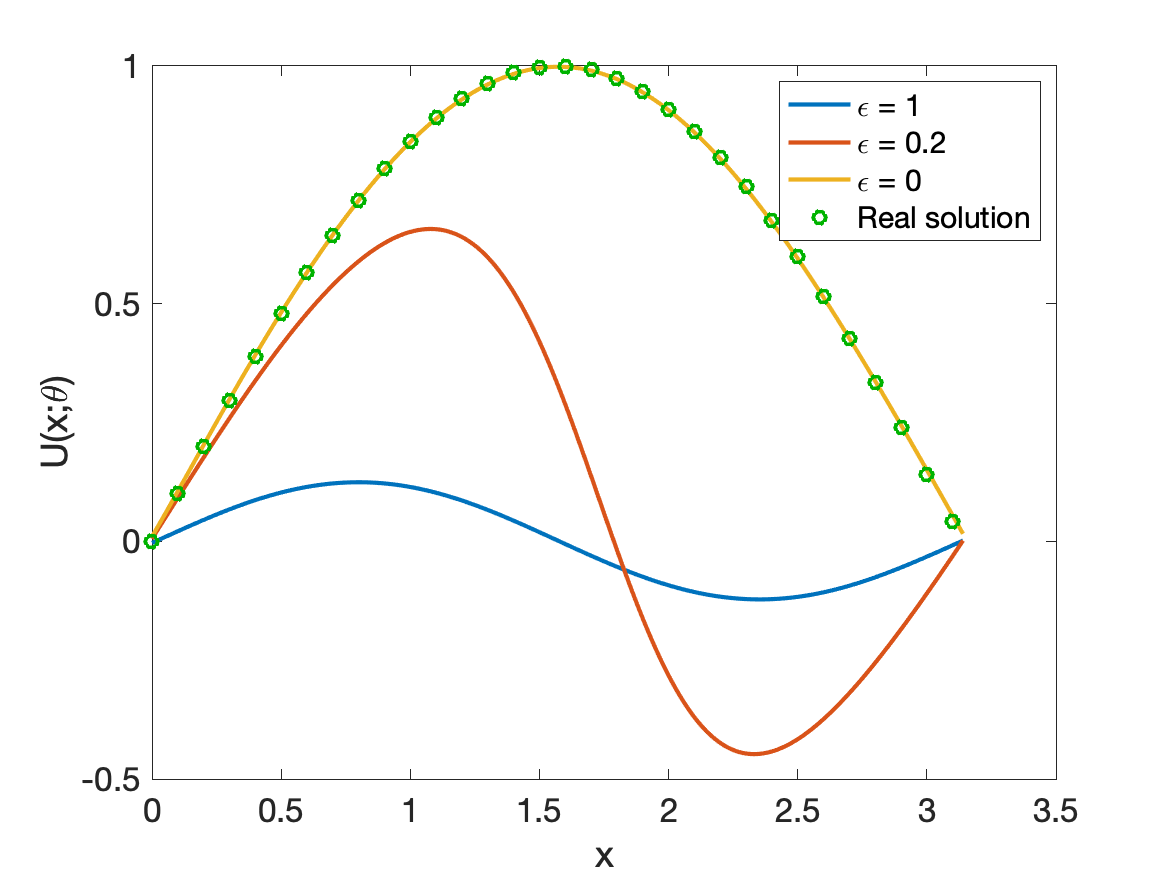}
			\caption{With a random initial guess} 
			\label{eg4_fig2}
		\end{subfigure}
		\caption{Numerical solutions of two tracking methods V.S. real solutions of system (\ref{example_3}). (Here we only plot solutions with $\epsilon=1, 0.2 \text{, and} 0$.)} 
	\end{figure}

	\begin{center}
	\begin{table}
	\begin{tabular}{|l|l|l|l|l|l|l|l|l|l|l|l|}
	\hline
	&\multicolumn{2}{|c|}{Homotopy tracking}&\multicolumn{2}{|c|}{With a random initial guess}\\
	\hline
	$\epsilon$&\# of iterations &Condition numbers&\# of iterations&Condition numbers\\
	\hline
	$1$&46&1.2e4&1867&1.4e3\\
	\hline
	$0.8$&11&1.2e4&73&8.6e3\\
	\hline
	$0.6$&5&2.3e4&1185&8.0e3\\
	\hline
	$0.4$&21&2.3e4&269&5.7e3\\
	\hline
	$0.2$&25&1.2e4&392&1.3e6\\
	\hline
	$0.1$&116&4.6e4&318&8.1e5\\
	\hline
	$0.05$&43&6.4e4&3471&8.1e4\\
	\hline
	$0.01$&285&3.1e8&371&2.0e5\\
	\hline
	$0$&2&1.0e10&57&5.0e4\\
	\hline
	\end{tabular}
	\caption{Numbers of iterations of the randomized Newton's method and condition numbers of (\ref{example_3}) when tracking $\epsilon$ from $1$ to $0$.}
	\label{table_3}
	\end{table}
	\end{center}

\subsection{2D Examples}

\subsubsection{An example with the analytical solution}
We consider the following 2D example on a rectangular domain $\Omega = [0,\pi]\times [0,\pi]$:
	\begin{equation}
	\label{example_4}
	\left\{
		     \begin{array}{lr}
		     \Delta u = -2u & \text{in } \Omega, \\
		     u(x,y) = \sin(x+y) & \text{on } \partial \Omega.
		     \end{array}
	\right.
	\end{equation}
Obviously, the real solution is
$u(x,y) = \sin(x+y).$ We use a one-hidden-layer neural network  with six nodes to
	approximate the solution, namely,
	\begin{equation}
		U((x,y)^T; \theta) = W_2\sigma(W_1(x,y)^T + b_1) + b_2,
	\end{equation}
	where $W_1\in \mathbb{R}^{6\times 2}$, $b_1\in \mathbb{R}^{6\times 1}$, $W_2\in \mathbb{R}^{1\times 6}$
	, $b_2\in \mathbb{R}$ and $\theta = \{W_1,b_1,W_2,b_2\}$.
The randomized Newton's method is used to solve Eq. (\ref{example_4}) with uniform sample points  on
	$[0, \pi]^2$ with step size 0.01. Figure \ref{eg3_fig1} plots the numerical solution versus the real solution, while Figure \ref{eg3_fig2} plots the numerical error which shows a good agreement with the real solution by using the randomized Newton's method.
	\begin{figure}
		\centering
		\begin{subfigure}{0.45\textwidth} 
			\includegraphics[width=\textwidth]{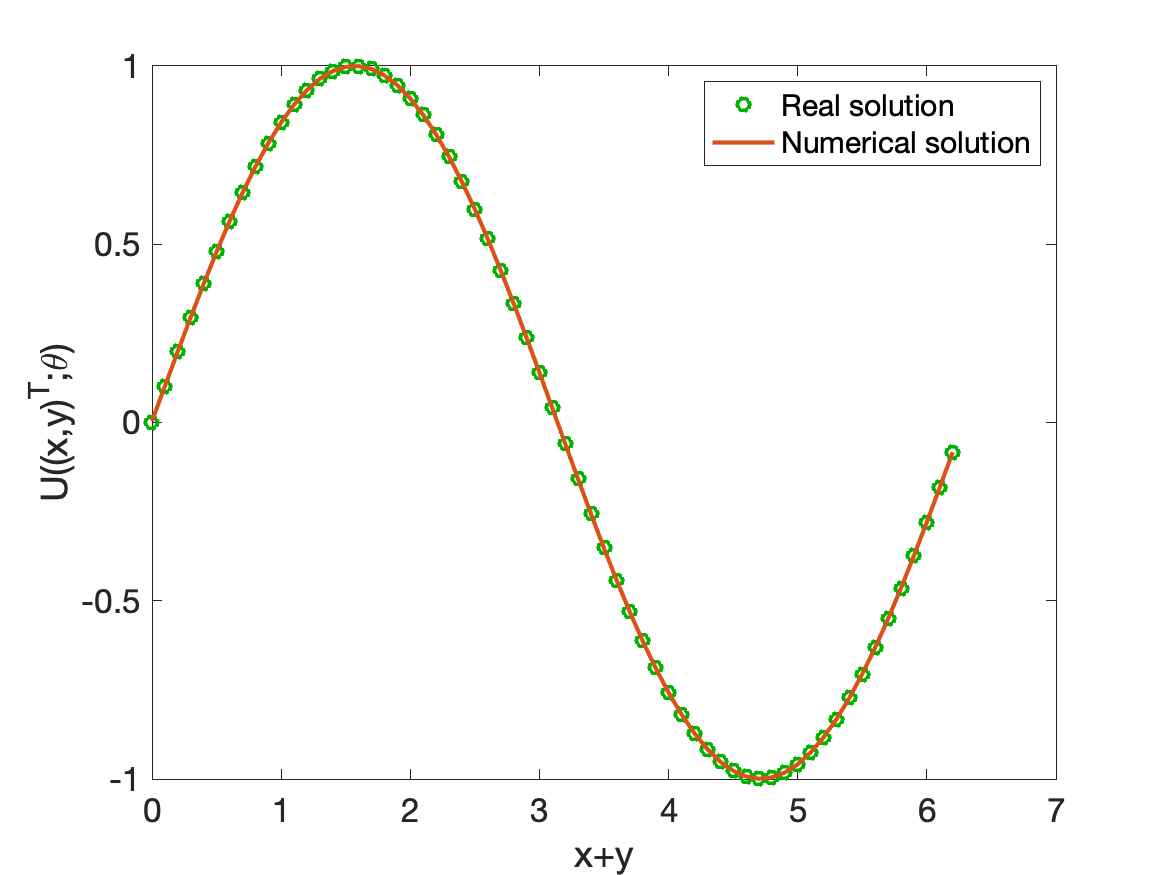}
			\caption{The numerical solution V.S. the real solution of Eq. (\ref{example_4}) in the $x+y$ direction.}
			\label{eg3_fig1}
		\end{subfigure}
		\vspace{1em} 
		\begin{subfigure}{0.45\textwidth} 
			\includegraphics[width=\textwidth]{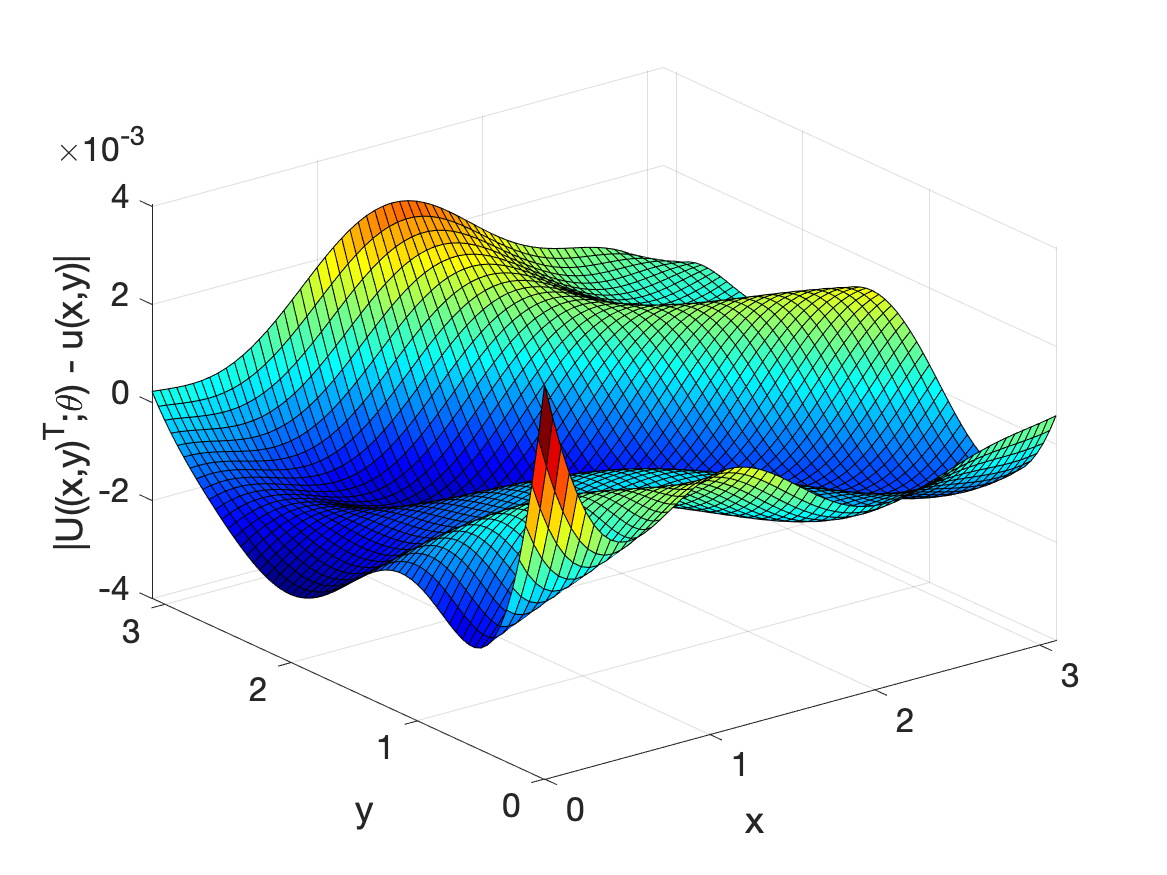}
			\caption{The numerical error of (\ref{example_4})} 
			\label{eg3_fig2}
		\end{subfigure}
		\caption{The numerical solution and error for solving Eq. (\ref{example_4}).} 
	\end{figure}

\subsubsection{The 2D Burger's equation}
	Next we consider the 2D Burger's equation with a viscosity term on $\Omega = \big[0,\frac{\pi}{\sqrt{2}}\big]\times\big[0,\frac{\pi}{\sqrt{2}}\big]$ \cite{chou2006high,chou2007high,hao2013homotopy}:
	\begin{equation}
	\label{example_5}
	\left\{
	     \begin{array}{lr}
	     (\frac{1}{\sqrt{2}}\frac{u^2}{2})_x + (\frac{1}{\sqrt{2}}\frac{u^2}{2})_y - \epsilon\Delta u = \sin(\frac{x+y}{\sqrt{2}})\cos(\frac{x+y}{\sqrt{2}}) & \text{in } \Omega, \\
	     u(x,y) = 0 & \text{on } \partial \Omega.
	     \end{array}
	\right.
	\end{equation}
Eq. (\ref{example_5}) recovers  the one-dimensional problem in Example 3 if we restrict the solution along the diagonal line. In order to approximate the solution, we use a one-hidden-layer neural network with three nodes with uniform sample points on $\big[0, \frac{\pi}{\sqrt{2}}\big]^2$ with step size 0.01.
Similar to Example 3, we use two tracking methods with respect to $\epsilon$ from $1$ to $0 $ coupled with the randomized Newton's method and list their numerical performance in	Table \ref{table_4} which shows that the randomized Newton's method converges faster with the homotopy tracking. The numerical solutions of two tracking methods are plotted in Figures \ref{eg5_fig1} and \ref{eg5_fig2} which show that the homotopy tracking yields the entropy solution while the random initial guess leads to an artificial steady state. In order to better compare with the real solution, here we plot solutions along
	the diagonal line of $\Omega$. The numerical errors are $1.8\times 10^{-2}$ for the solution with the homotopy tracking and $1.6\times 10^{-2}$ for the solution with a random initial guess.

	\begin{figure}
		\centering
		\begin{subfigure}{0.45\textwidth} 
			\includegraphics[width=\textwidth]{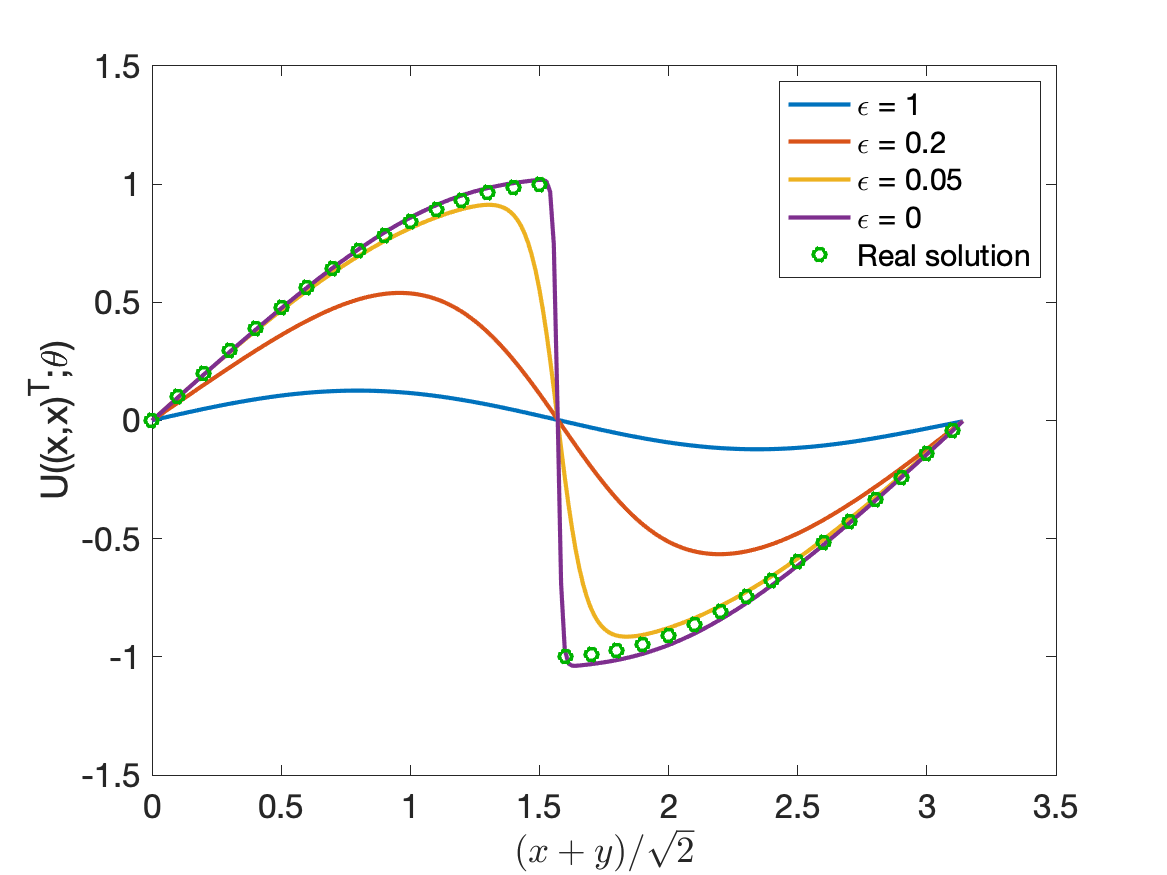}
			\caption{Homotopy tracking} 
			\label{eg5_fig1}
		\end{subfigure}
		\vspace{1em} 
		\begin{subfigure}{0.45\textwidth} 
			\includegraphics[width=\textwidth]{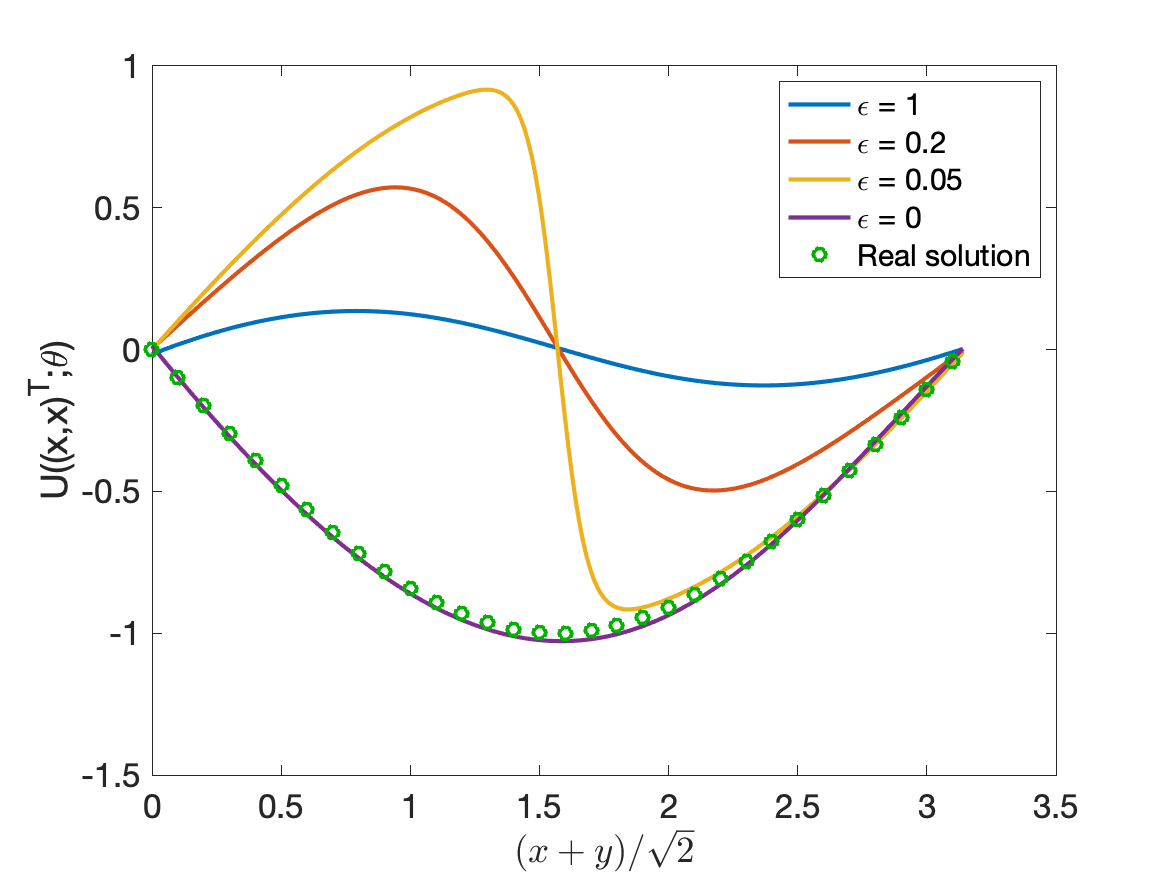}
			\caption{With a random initial guess} 
			\label{eg5_fig2}
		\end{subfigure}
		\caption{Numerical solutions of two tracking methods V.S. real solutions
of system (\ref{example_5}). (Here we only plot solutions with $\epsilon=1, 0.2, 0.05 \text{ and } 0$.)}
	\end{figure}

	\begin{center}
	\begin{table}
	\begin{tabular}{|l|l|l|l|l|l|l|l|l|l|l|}
	\hline
	&\multicolumn{2}{|c|}{Homotopy tracking}&\multicolumn{2}{|c|}{With a random initial guess}\\
	\hline
	$\epsilon$&$\#$ of iterations&Condition numbers&$\#$ of iterations&Condition numbers\\
	\hline
	$1$&75&3.4e2&149&5.0e4 \\
	\hline
	$0.8$&15&5.2e2&25&1.1e4 \\
	\hline
	$0.6$&8&5.2e2&119&1.6e3 \\
	\hline
	$0.4$&30&1.5e4&66&4.0e3 \\
	\hline
	$0.2$&22&5.2e3&32&3.8e4 \\
	\hline
	$0.05$&182&1.5e3&337&8.4e4 \\
	\hline
	$0.02$&233&1.1e5&494&2.5e4 \\
	\hline
	$0.01$&126&2.6e4&137&1.7e3 \\
	\hline
	$0$&928&1.7e6&176&3.2e4 \\
	\hline
	\end{tabular}
	\caption{Numbers of iterations of the randomized Newton's method and condition numbers of (\ref{example_5}) when tracking $\epsilon$ from $1$ to $0$.}
	\label{table_4}
	\end{table}
	\end{center}

\subsection{A high-dimensional example}
	We consider the following Laplace's  equation on the unit $n$-ball, namely, $\Omega=\{x\in  \mathbb {R}^n : \|x\|\leq 1\}$
	\begin{equation}
	\label{example_6}
	\left\{
	     \begin{array}{lr}
	     -\Delta u = \|x\| & \text{in } \Omega, \\
		     u(x) = 1 & \text{on } \partial \Omega. \\
	     \end{array}
	\right.
	\end{equation} For the radial symmetric case, the real solution has the following form
	\begin{equation}
		u(r;n) = \frac{-r^3+3n+4}{3n+3},
	\end{equation}
	where $0\le r\le 1$ and $n\ge 2$.

Numerically, we use one-hidden-layer neural networks with different nodes and $10^{2n}$
	uniform sample points on $\Omega$ to solve Eq. (\ref{example_6}) by employing the randomized Newton's method.
	Figure \ref{eg6_fig} shows a good agreement of numerical solutions with real solutions for $n=2$ and $n=3$. For higher dimensional case ($n>3$), we list numerical errors  different $n$ in Table \ref{table_5}, which shows that the randomized Newton's method can be used for solving high-dimensional differential equations.

	\begin{figure}[htbp]
        \includegraphics[width=.32\textwidth]{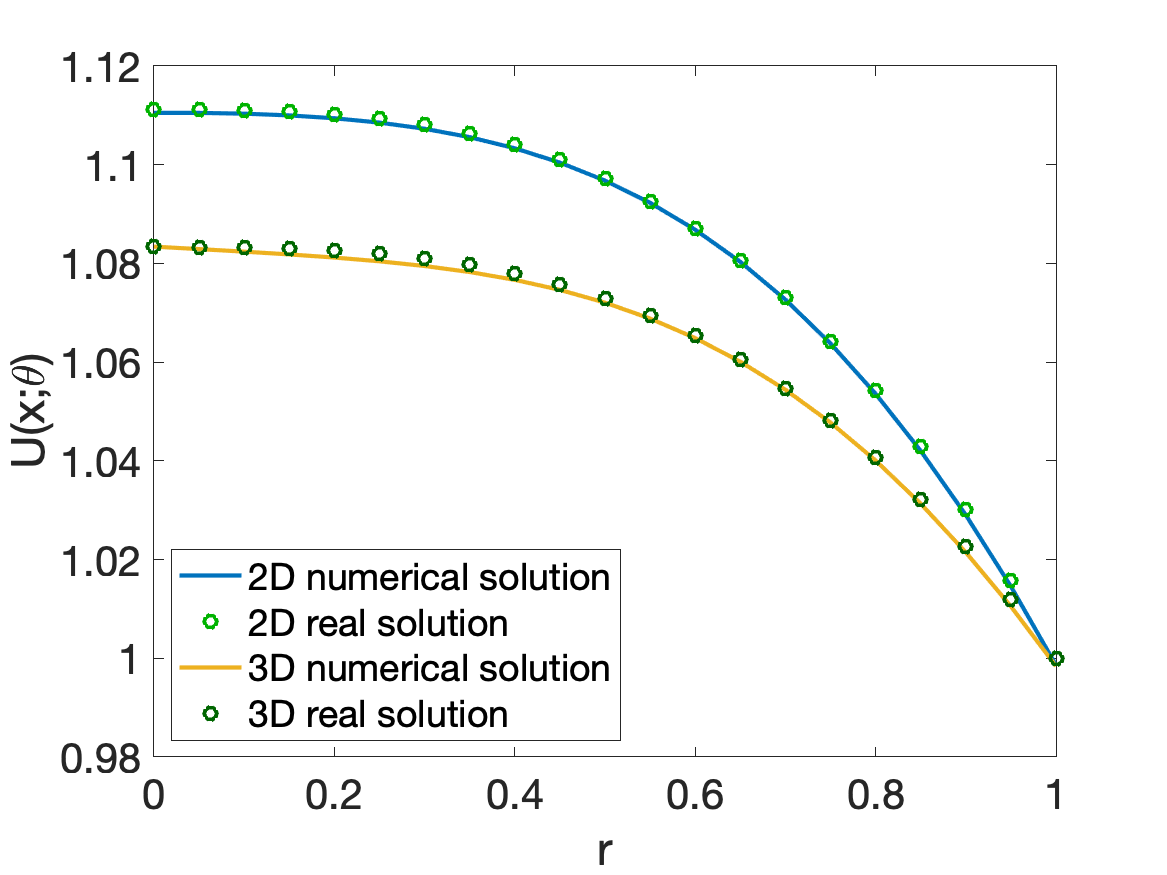}
        \includegraphics[width=.32\textwidth]{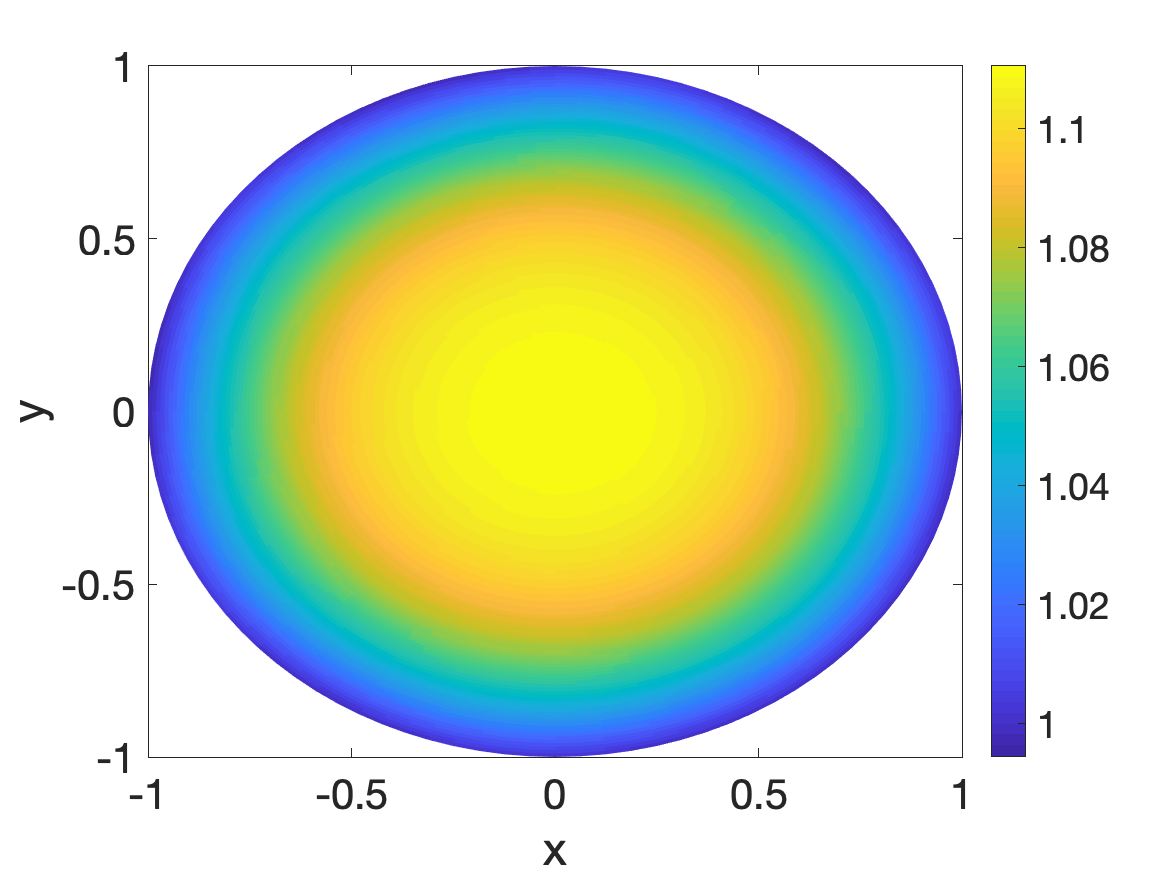}
         \includegraphics[width=.32\textwidth]{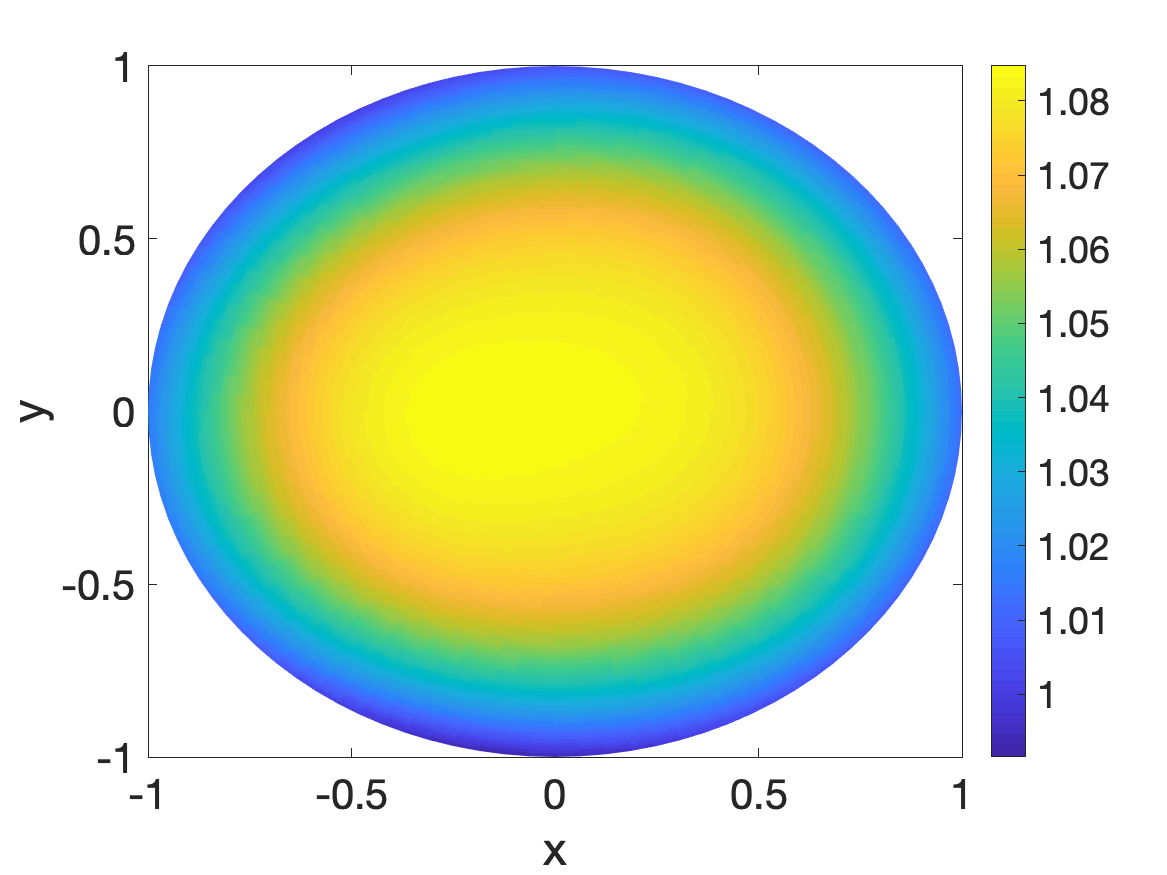}
        \caption{{\bf Left:} numerical solutions vs. real solutions of system (\ref{example_6}) in the radial direction; {\bf Middle:} the numerical solution for $n=2$; {\bf Right:} the numerical solution on the $x-y$ plane for $n=3$.}
       	\label{eg6_fig}
        \end{figure}

	\begin{table}
\centering
\begin{tabular}{|c|c|c|c|}
\hline
$n$&Width of hidden layers&$\#$ of variables& Errors\\
\hline
2&10&41&1.1e-3\\
\hline
3&35&176&2.1e-3\\
\hline
4&80&481&5.0e-3\\
\hline
5&100&701&4.2e-3\\
\hline
6&100&801&4.1e-3\\
\hline
\end{tabular}
\caption{Numerical errors of different $n$ for solving Eq. (\ref{example_6}).}
\label{table_5}
\end{table}

\subsection{An application to the pattern formation}
The pattern formation, as an important problem in physics and biology, involves nonlinear differential equations in various mathematical models. One of the key questions is to compute the nonuniform steady states of nonlinear differential equations, which are the so-called the stationary spatial patterns  \cite{kondo2010reaction,gierer1972theory,maini2012turing,turing1990chemical}.  However, numerical methods of solving these nonlinear systems, e.g, Newton's method,  are normally sensitive to the initial guesses that are hard to construct for the pattern formation models. In this paper, we will use the neural network discretization to ``learn" these nonuniform patterns. Although the accuracy of neural network discretization is low, it will provide an alternative way to compute the multiple solutions, which is hard for the traditional discretizations such as the finite difference or finite element methods. We use the Gray-Scott model \cite{pearson1993complex,lee1993pattern} to illustrate the idea:
	\begin{equation}
		\label{gray-scott}
		\begin{aligned}
			&\frac{\partial A}{\partial t} = D_A\Delta A + SA^2 - (\mu+\rho)A,\\
			&\frac{\partial S}{\partial t} = D_S\Delta S - SA^2 + \rho(1-S),
		\end{aligned}
	\end{equation}
	where $A$ is the concentration of an activator and $S$ is the
	concentration of a substrate. The growth of $A$ reacts with $S$ fed from the activator with a rate $\rho$, and $S$ is
converted to an inert product at the rate $\mu$. $D_A$ and $D_S$ are
the diffusion coefficients of $A$ and $S$, respectively. In our computations,  we choose $D_A = 2.5\times 10^{-4}$, $D_S = 5\times 10^{-4}$,
	$\rho = 0.04$ and $\mu = 0.065$.

First, we consider the 1D case with $[0,1]$ as the domain and use a one-hidden-layer neural network with ten nodes to discretize $A$ and $S$, namely,
	\begin{equation}
		\label{2d_network}
		\begin{aligned}
		A(x) &= W_2 \sigma(W_1 x + b_1) + b_2 \hbox{,~} 		S(x) &= W_4 \sigma(W_3 x + b_3) + b_4.
		\end{aligned}
	\end{equation}
Based on the following four initial guesses,  we obtained different steady patterns by using the randomized Newton's method and show in Figure \ref{eg8}:
	\begin{equation}
	\begin{aligned}
	&\text{ IG1: } (A,S) =\left(\frac{3}{10}\cos(3\pi x)+\frac{1}{2},-\frac{3}{10}\cos(3\pi x)+\frac{1}{2}\right),\\
	&\text{ IG2: } (A,S) =\left(-\frac{3}{10}\cos(3\pi x)+\frac{1}{2},\frac{3}{10}\cos(3\pi x)+\frac{1}{2}\right),\\
	 &\text{ IG3: } (A,S) =\left(\frac{3}{10}\cos(\pi x)+\frac{1}{2},-\frac{3}{10}\cos(\pi x)+\frac{1}{2}\right),\\
	 &\text{ IG4: } (A,S) =\left(-\frac{3}{10}\cos(\pi x)+\frac{1}{2},\frac{3}{10}\cos(\pi x)+\frac{1}{2}\right).
	 \end{aligned}\label{IG}
	 \end{equation}
	\begin{figure}
		\centering
		\begin{subfigure}{0.45\textwidth} 
			\includegraphics[width=\textwidth]{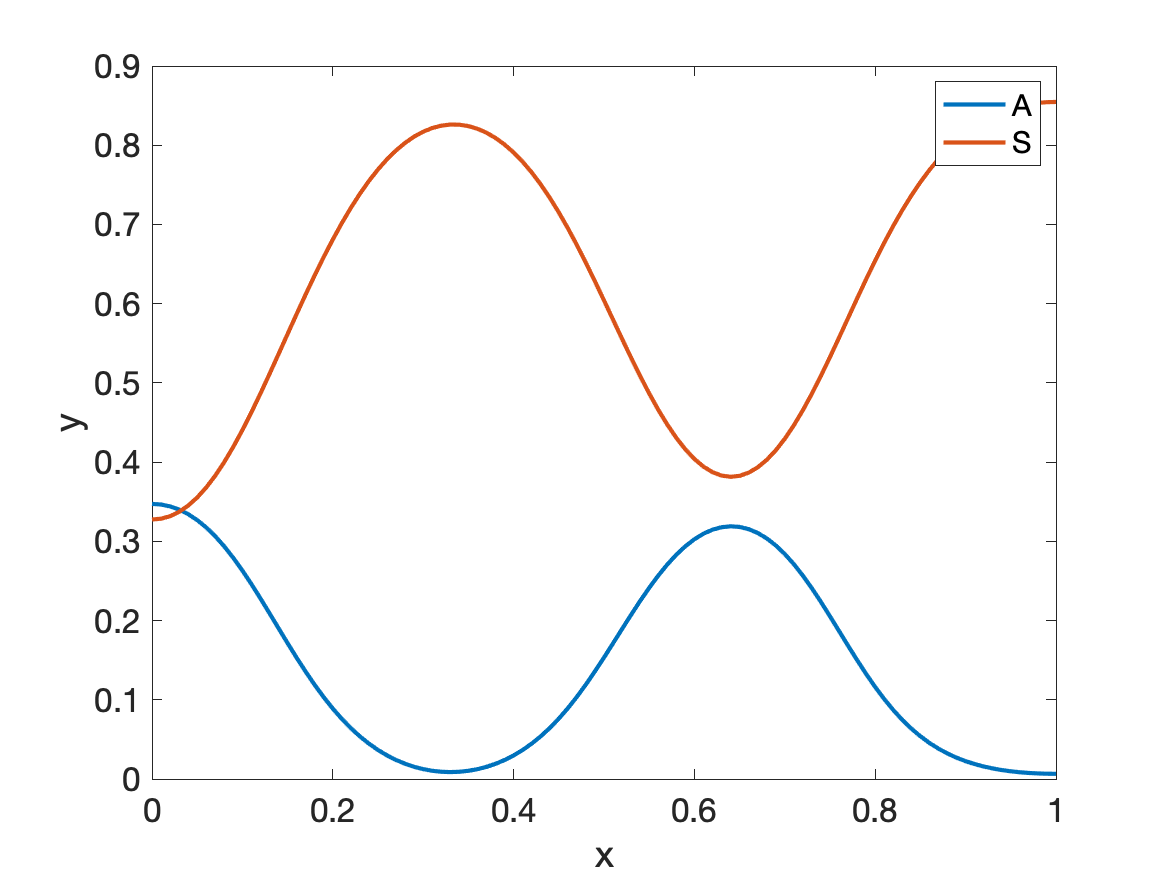}
			\caption{Steady patterns by IG1} 
			\label{eg8_1}
		\end{subfigure}
		\vspace{1em} 
		\begin{subfigure}{0.45\textwidth} 
			\includegraphics[width=\textwidth]{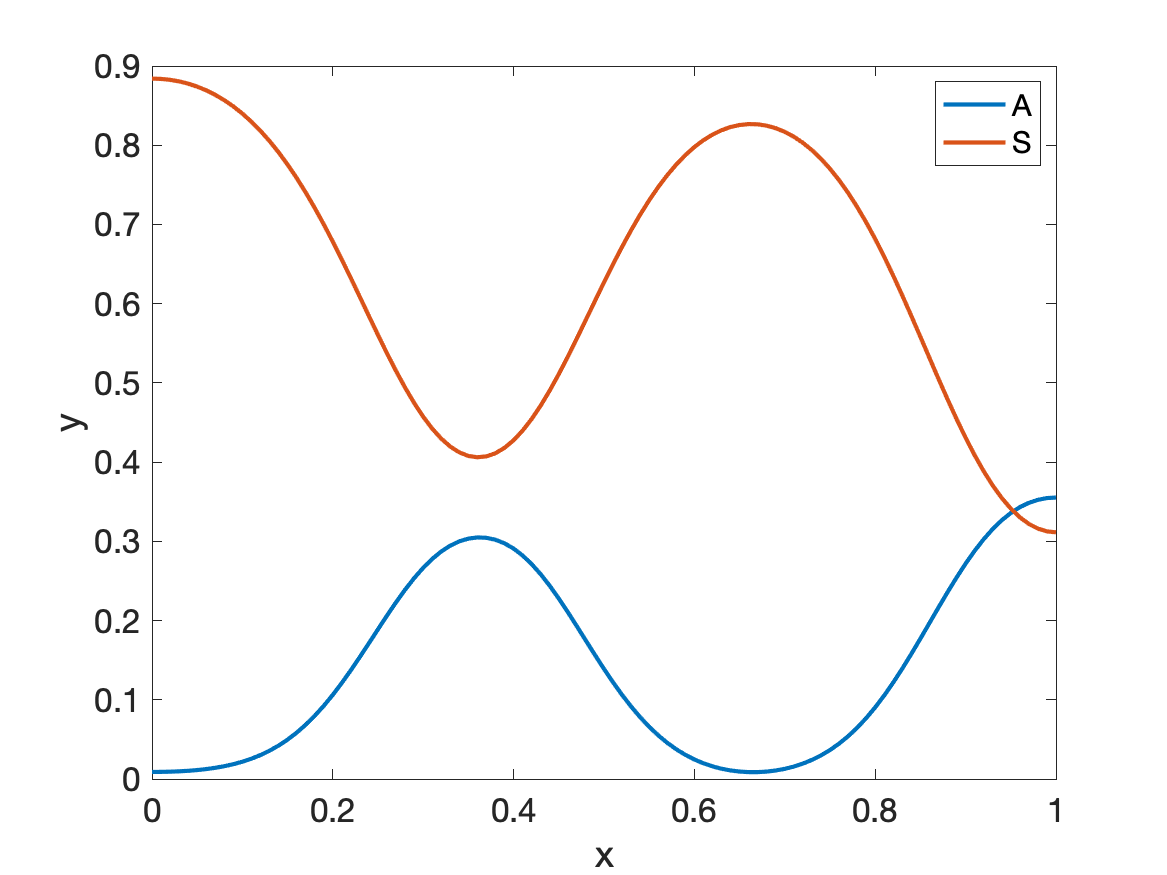}
			\caption{Steady patterns by  IG2} 
			\label{eg8_2}
		\end{subfigure}
		\vspace{1em} 
		\begin{subfigure}{0.45\textwidth} 
			\includegraphics[width=\textwidth]{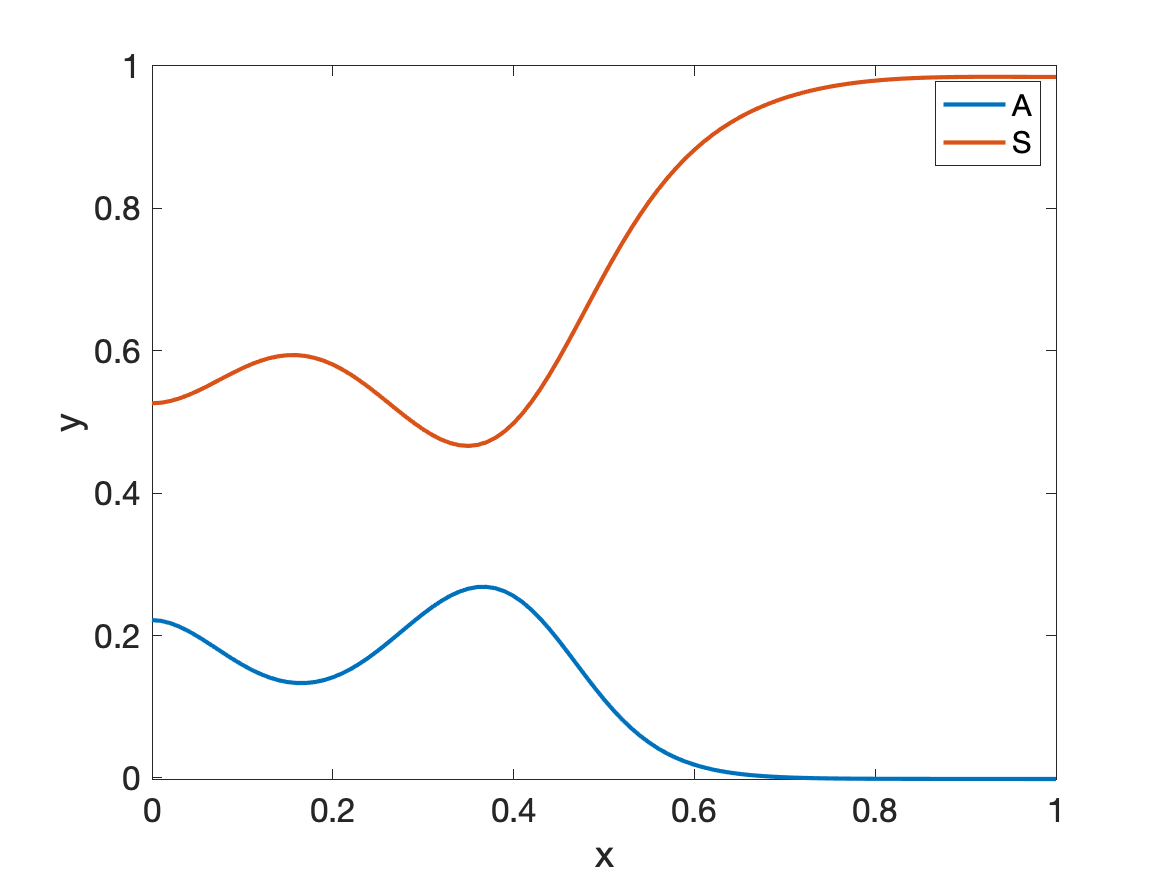}
			\caption{Steady patterns by  IG3} 
			\label{eg8_3}
		\end{subfigure}
		\vspace{1em} 
		\begin{subfigure}{0.45\textwidth} 
			\includegraphics[width=\textwidth]{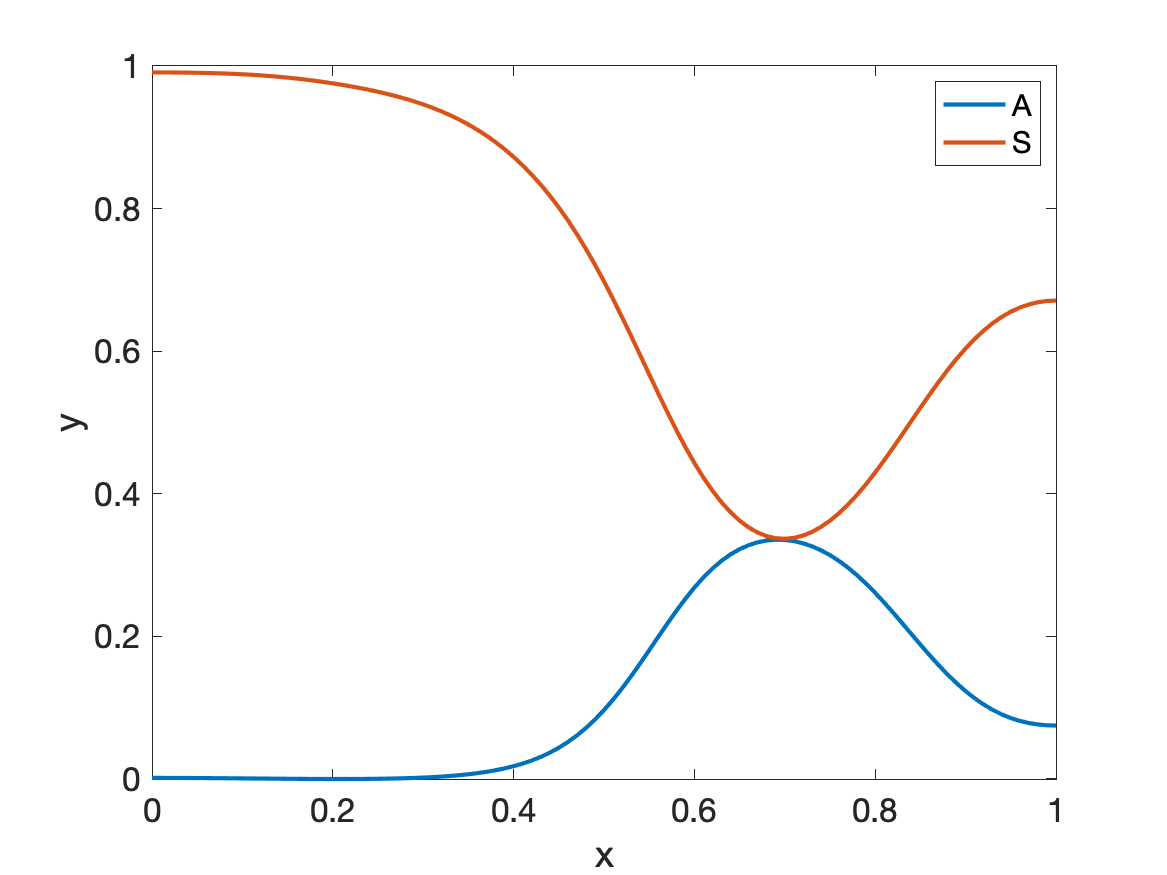}
			\caption{Steady patterns by  IG4} 
			\label{eg8_4}
		\end{subfigure}
		\caption{Steady patterns of the 1D Gray-Scott model by using differential initial guesses shown in Eq. (\ref{IG}).}\label{eg8} 
	\end{figure}
	Secondly, we consider the 2D case with the domain as $[0,1]^2$  and also use a one-hidden layer neural network with ten nodes as the discretization, namely,
	\begin{equation}
		\label{2d_network}
		\begin{aligned}
		A(x,y) &= W_2 \sigma(W_1 (x,y)^T + b_1) + b_2	\hbox{~and~}	S(x,y) &= W_4 \sigma(W_3 (x,y)^T + b_3) + b_4.
		\end{aligned}
	\end{equation}
In the 2D case, we run the randomized Newton's method many times with the same initial guess in order to ``learn" the multiple steady patterns. For instance, the initial guess shown in
 Figure \ref{gray_scott_1} (left) yields two stable patterns shown in  Figure \ref{gray_scott_1} (right); Figure  \ref{gray_scott_2} shows four steady patterns can be ``learned" from one initial guess. Thus the randomized Newton's method can be used to compute the multiple solutions of nonlinear differential equations.

	\begin{figure}[htbp]
        \centering
        \begin{minipage}[t]{0.5\textwidth}
        \centering
        \vspace{-7em}
        \includegraphics[width=\textwidth]{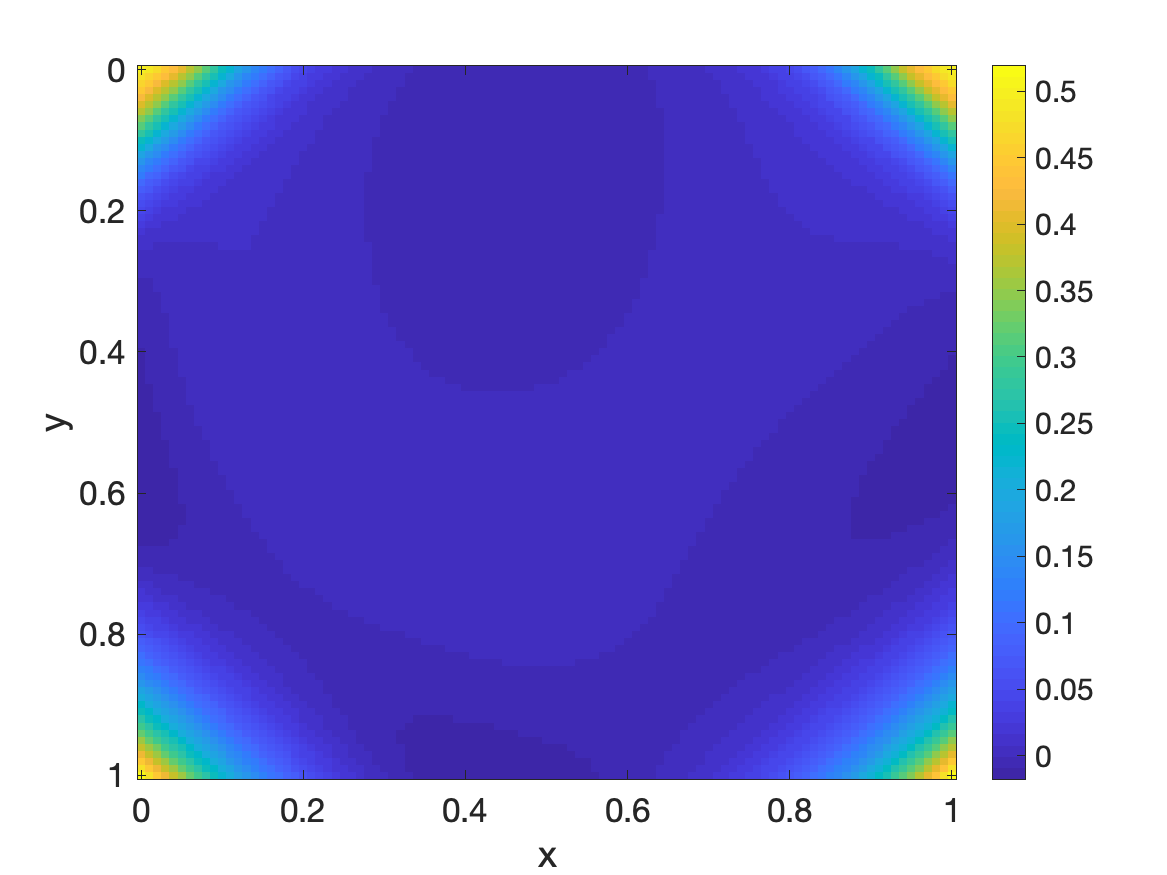}
        \end{minipage}
        \begin{minipage}[t]{0.4\textwidth}
        \centering
        \includegraphics[width=\textwidth]{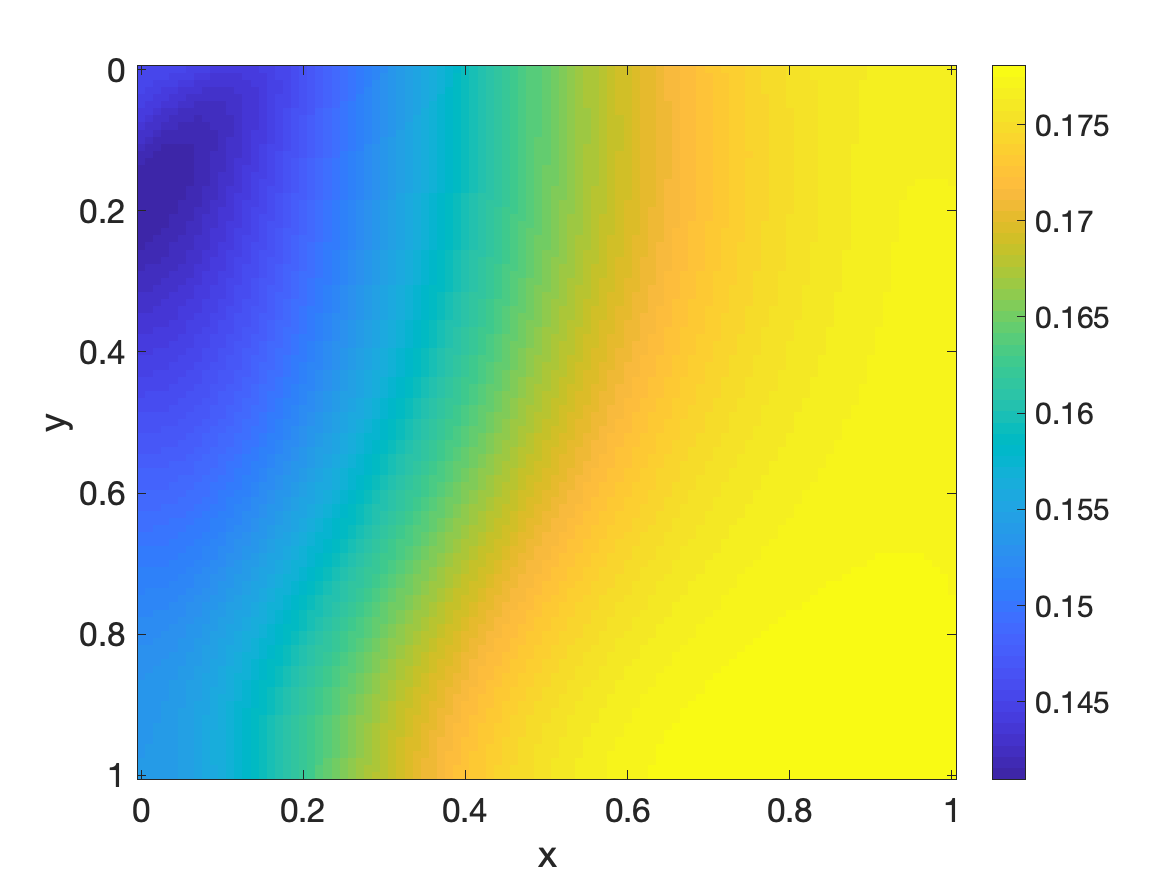}\\
         \includegraphics[width=\textwidth]{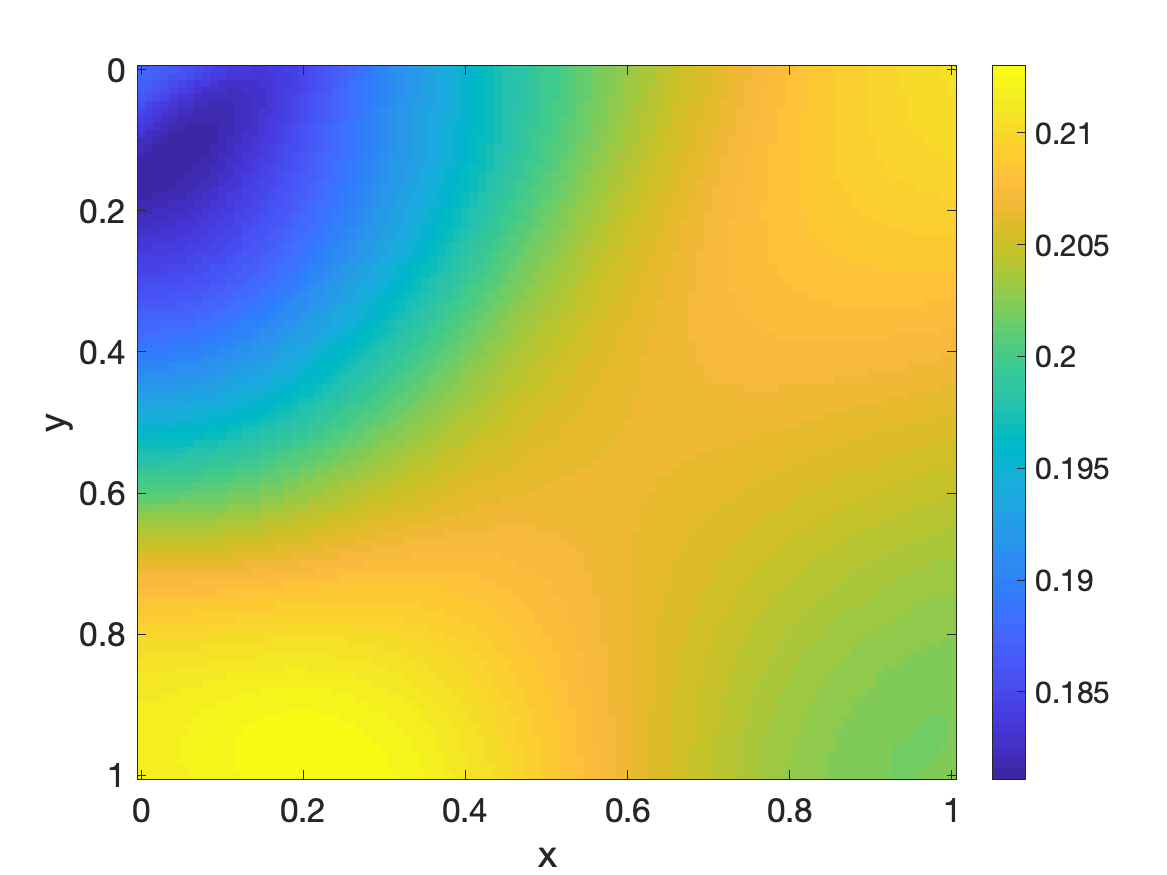}
        \end{minipage}
        \caption{{\bf Left:} The initial guess of $A(x,y)$; {\bf Right:} steady patterns of $A(x,y)$ for the 2D Gray-Scott model.}
        \label{gray_scott_1}
        \end{figure}

        	\begin{figure}[htbp]
        \centering
        \begin{minipage}[t]{0.4\textwidth}
        \centering
       \vspace{-6em}
        \includegraphics[width=\textwidth]{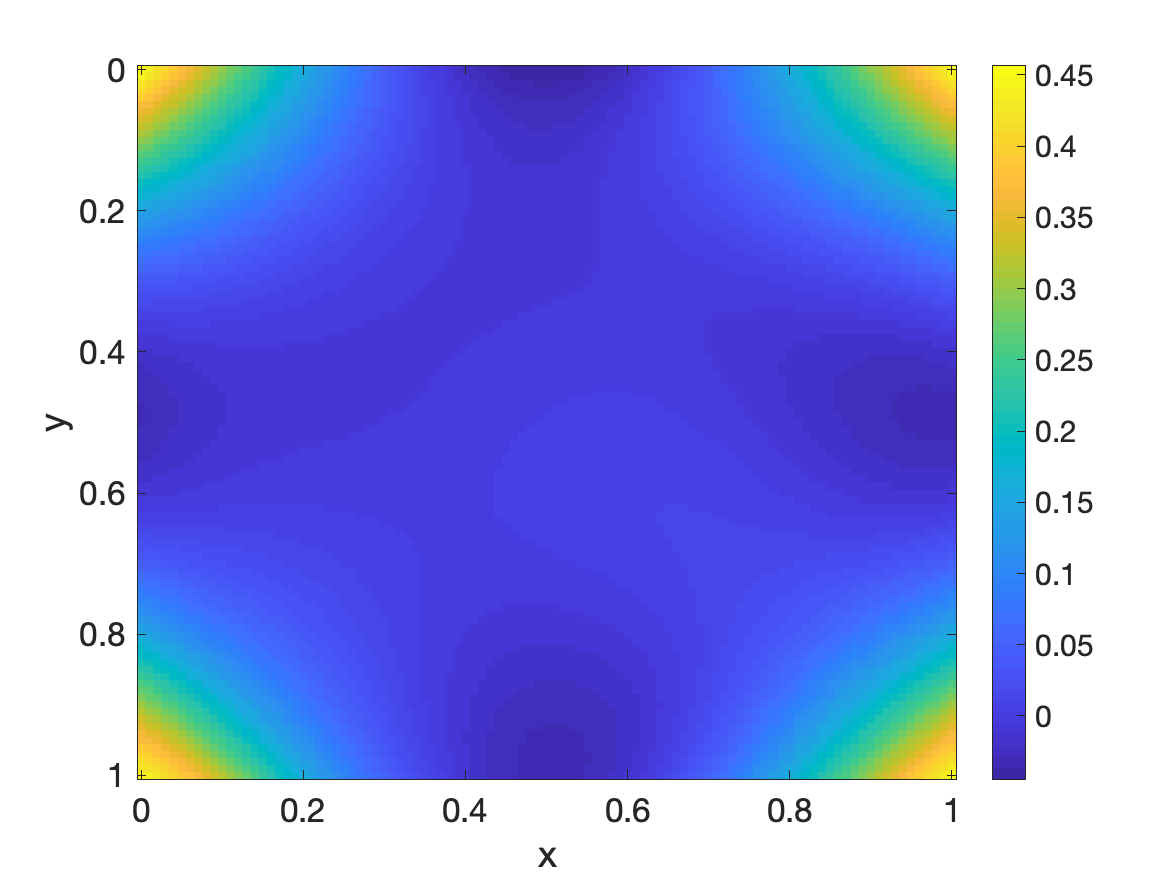}
        \end{minipage}
        \begin{minipage}[t]{0.25\textwidth}
        \centering
        \includegraphics[width=\textwidth]{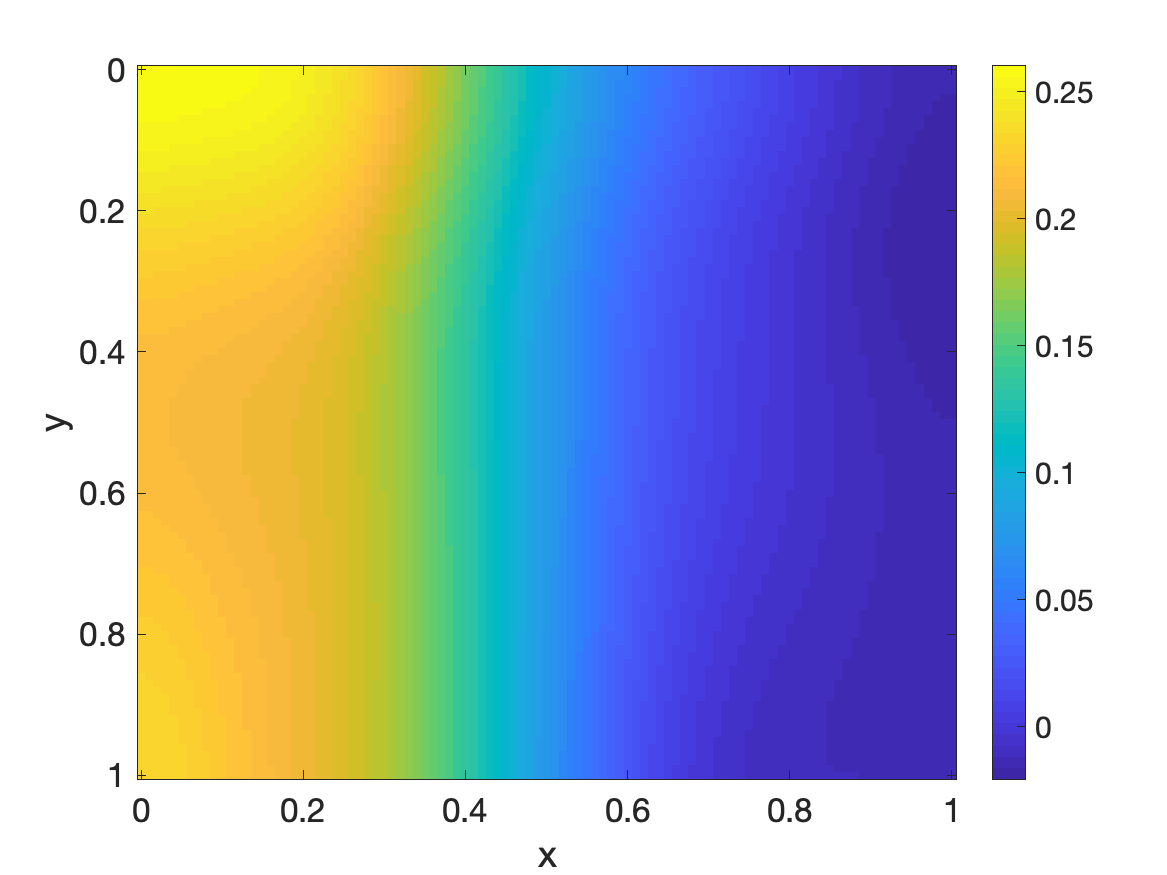}
        \includegraphics[width=\textwidth]{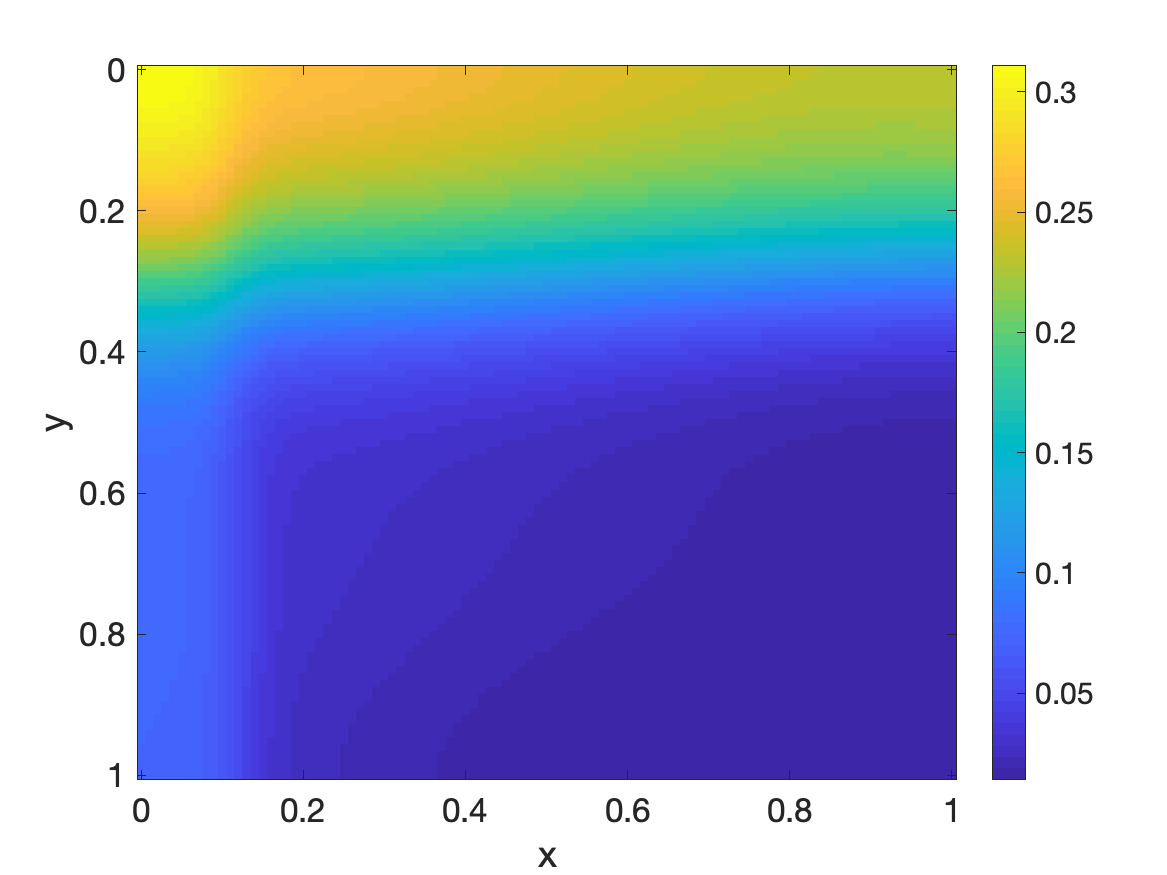}
        \end{minipage}
        \begin{minipage}[t]{0.25\textwidth}
        \centering
        \includegraphics[width=\textwidth]{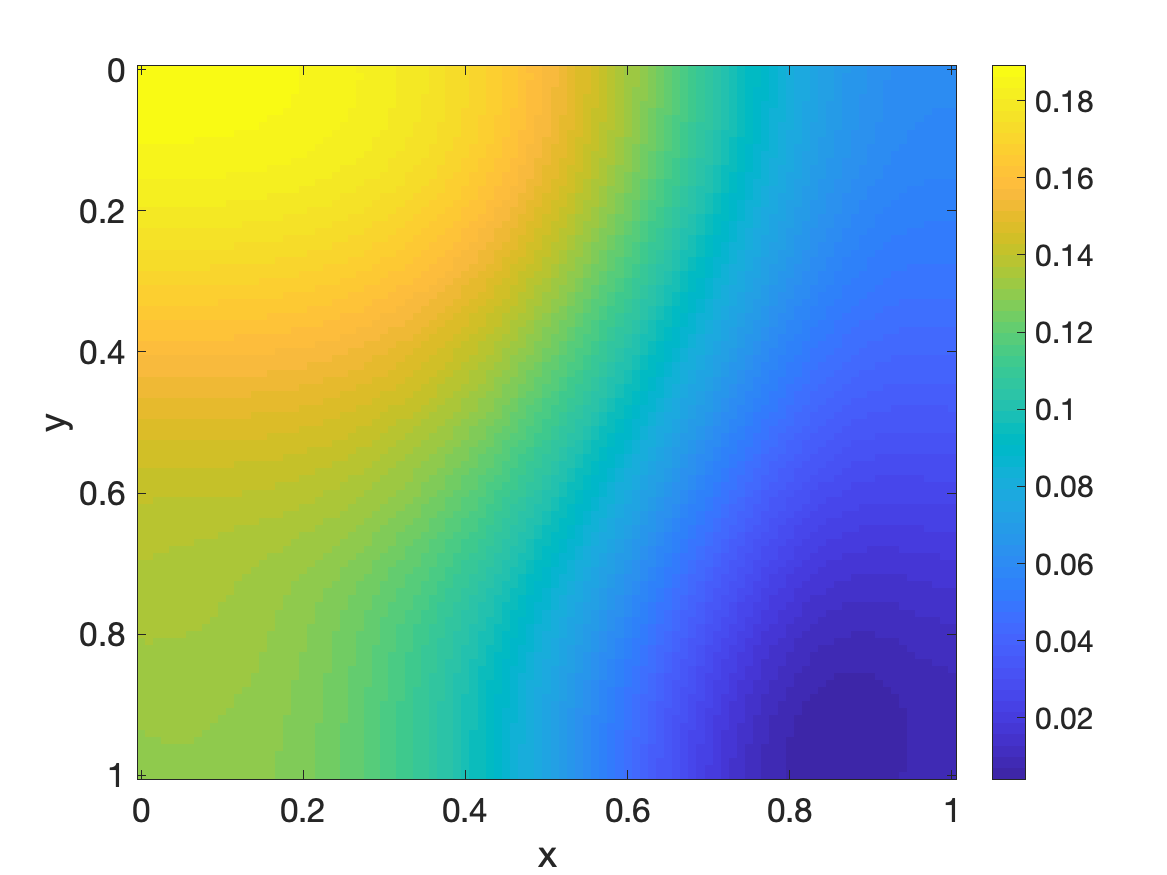}
        \includegraphics[width=\textwidth]{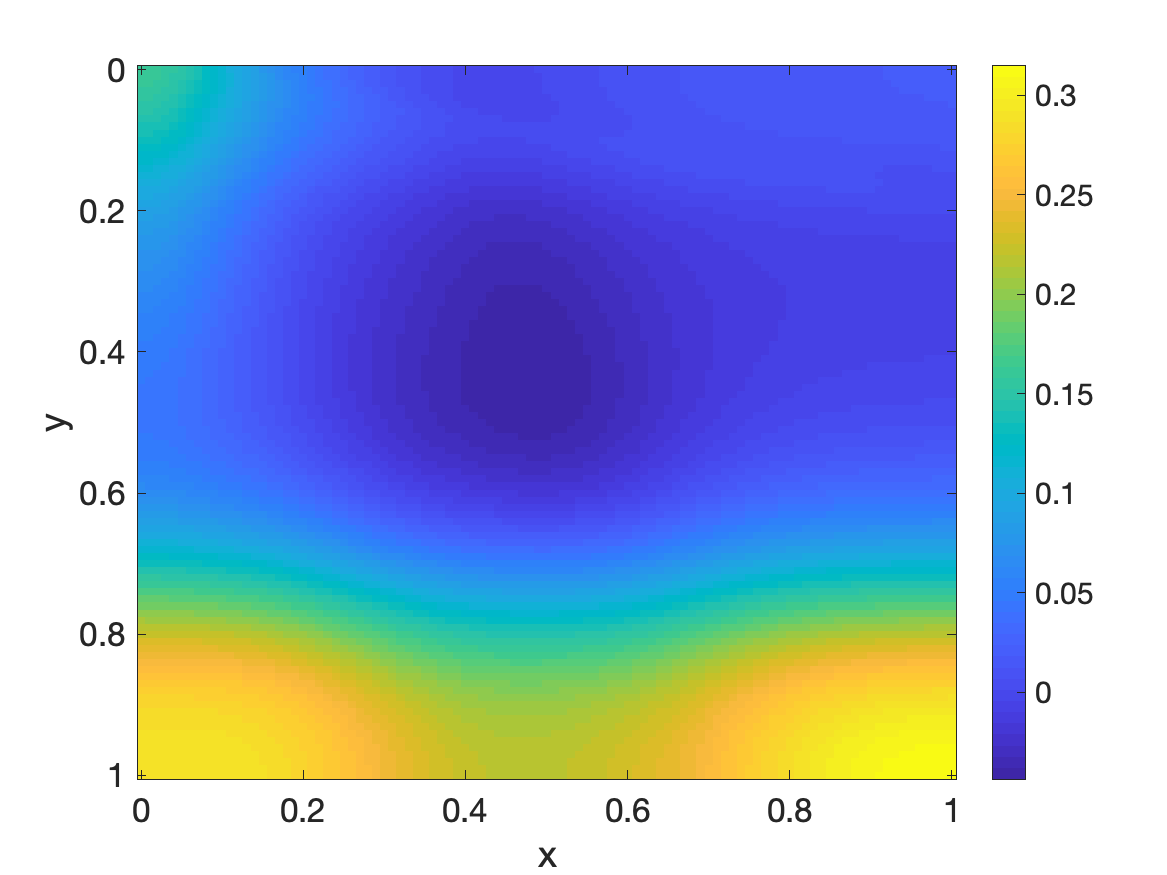}
        \end{minipage}
        \caption{{\bf Left:} The initial guess of $A(x,y)$; {\bf Right:} steady patterns of $A(x,y)$ for the 2D Gray-Scott model.}
        \label{gray_scott_2}
        \end{figure}

\section{Conclusion}
In this paper, we develop a randomized Newton's method for solving differential equations based on the fully connected neural network discretization. This proposed method is designed specifically to solve an overdetermined nonlinear system, since the number of sample points in such a system is much larger than the number of variables. For each iteration, we randomly choose equations from the nonlinear system
and apply the classical Newton's method repeatedly, and we prove theoretically that the randomized Newton's method has a local quadratic convergence. Using several examples, we also demonstrate, numerically, that the randomized Newton's method for solving both linear and nonlinear equations is indeed efficient and feasible. Moreover, the method developed here can be used to solve high-dimensional differential equations that are otherwise hard to solve by traditional numerical methods. Another advantage of this method is that it allows for computing the multiple solutions of nonlinear differential equations, such as pattern formation problems. In future work, we will apply the other types of neural networks to discretize differential equations (e.g., conventional neural networks) and aim to reduce the redundancy of neural network discretization, in order to improve the convergence of the randomized Newton's method.

\bibliographystyle{plain}
\bibliography{references}

\end{document}